\documentclass[reqno,10pt]{amsart}
\usepackage[utf8]{inputenc}
\usepackage{a4wide,amsmath,amsfonts,amssymb,amsthm}
\usepackage{mathtools,stmaryrd,bbm}
\usepackage{subfiles,appendix}
\usepackage{tikz}
\usepackage{hyperref}   

\title[Asymptotic stability of a finite sum of solitary waves for the ZK equation]{Asymptotic stability of a finite sum of solitary waves for the Zakharov-Kuznetsov equation}

\author[D. Pilod]{Didier Pilod}
\address{Department of Mathematics, University of Bergen, Postbox 7800, 5020 Bergen, Norway}
\email{Didier.Pilod@uib.no}

\author[F. Valet]{Fr\'ed\'eric Valet}
\address{CY Cergy Paris Universit\'e. Laboratoire de recherche Analyse, G\'eom\'etrie, Mod\'elisation (UMR CNRS 8088), 2 avenue Adolphe Chauvin, 95302 Cergy-Pontoise Cedex, France}
\email{Frederic.Valet@cyu.fr}

\newtheorem{lemm}{Lemma}[section]
\newtheorem{theo}[lemm]{Theorem}

\newtheorem{propo}[lemm]{Proposition}

\newtheorem{claim}[lemm]{Claim}

\newtheorem{rema}[lemm]{Remark}

\numberwithin{equation}{section}

\newcommand{\beone}{\mathbf{e}_1}

\newcommand{\bx}{\mathbf{x}}
\newcommand{\by}{\mathbf{y}}
\newcommand{\bz}{\mathbf{z}}

\usepackage{comment}
\excludecomment{toexclude}

\usetikzlibrary{math} 
\tikzmath{\Num=200;}

\date{\today}

\begin{document}
	
	\maketitle
	
	\begin{abstract}
		We prove the asymptotic stability of a finite sum of well-ordered solitary waves for the Zakharov-Kuznetsov equation in dimensions two and three. Moreover, we derive a qualitative version of the orbital stability result which turns out to be useful for the study of the collision of two solitary waves.  
		
		The proof extends the ideas of Martel, Merle and Tsai for the sub-critical gKdV equation in dimension one to the higher-dimensional case. It relies on  monotonicity properties on oblique half-spaces and rigidity properties around one solitary wave introduced by C\^ote, Mu\~noz, Pilod and Simpson in dimension two, and by Farah, Holmer, Roudenko and Yang in dimension three.
	\end{abstract}
	
	\section{Introduction}
	
	\subsection{The Zakharov-Kuznetsov equation}
	In this paper, we consider the Zakharov-Kuzentsov (ZK) equation
	\begin{align}\label{ZK}
		\partial_t u + \partial_x \left( \Delta u + u^2 \right)=0, 
	\end{align}
	where $u=u(t,\bx)$ is a real-valued function, $t \in \mathbb R$, $\bx =(x,\by) \in \mathbb R \times \mathbb R^{d-1}$, $d \ge 2$ is an integer, and $\Delta=\partial_x^2+\Delta_{\by}$ denotes the Laplacian in $\mathbb R^d$. This equation is a natural high-dimensional generalization of the Korteweg-de Vries (KdV) equation, which corresponds to the case of spatial dimension $d=1$. In dimensions $d=2$ and $3$, it describes the propagation of ionic-acoustic waves in a magnetized cold plasma (see  equation (15.175) in \cite{Bellan}). Equation \eqref{ZK} was introduced by Zakharov and Kuznetsov \cite{ZK} in the $3$-dimensional case. The rigorous derivation of the ZK equation from the Euler-Poisson system with magnetized field in the long wave regime was carried out by Lannes, Linares and Saut in \cite{LaLiSa}. We also refer to the work of Han Kwan \cite{HanKwan} for the derivation of ZK from the Vlasov-Poisson system. 
	
	Contrary to the KdV equation in one dimension, the Zakharov-Kuznetsov equation is not known to be completely integrable in dimensions $d \ge 2$. Nevertheless, it enjoys a Hamiltonian structure, where the Hamiltonian (also called energy) is given by
	\begin{equation} \label{def:energy}
		E(u)(t) := \int_{\mathbb{R}^d} \left( \frac{\vert \nabla u(t,\bx) \vert^2}{2} - \frac{u(t,\bx)^3}{3} \right) d\bx .
	\end{equation}
	The functionals $\int_{\mathbb{R}^d} u(t,\bx) d\bx$ and 
	\begin{equation} \label{def:mass}
		M(u)(t) := \int_{\mathbb{R}^d} u(t,\bx)^2 d\bx
	\end{equation}
	are also, at least formally, conserved by the flow of ZK. 
	Moreover, note that the set of solutions of \eqref{ZK} is invariant under scaling and translation. More precisely, if $u$ is a solution of \eqref{ZK}, then, for any $c>0$ and $\bz \in \mathbb R^d$, $u_{c,\bz}$ is also a solution, where
	\begin{align*}
		u_{c,\bz}(t,\bx) := c u (c^\frac32 t, c^{\frac12} (\bx-\bz)).
	\end{align*}
	Here we will focus on the two and three dimensional cases which are $L^2$-subcritical. 
	
	The Cauchy problem associated to the ZK equation has been extensively studied in the recent years. In dimension $d=2$, global well-posedness has been proved in the energy space $H^1(\mathbb R^2)$ by Faminski\u{\i} \cite{Fam}. We also refer to \cite{LiPa,GruHe,MoPi,Kino} for more results at lower regularity. To prove global well-posedness in the energy space in dimension $d=3$ is more challenging and the problem has only been solved recently by Herr and Kinoshita \cite{HeKin} (see \cite{LiSa,RiVe,MoPi} for former results). Below, we summarize the results in \cite{Fam,HeKin}. Let $d=2$ or $d=3$. Then, for any $u_0 \in H^1(\mathbb R^d)$, there exists a unique\footnote{The uniqueness only holds in a subspace of $C(\mathbb R : H^1(\mathbb R^d))$} solution $u$ of \eqref{ZK} in $C(\mathbb R : H^1(\mathbb R^d))$. Moreover, $M(u)(t)=M(u_0)$ and $E(u)(t)=E(u_0)$, for all $t \in \mathbb R$. 
	
	\subsection{Solitary waves}
	The Zakharov-Kuznetsov admits a family of solitary wave solutions propagating at constant speed in the horizontal direction. These solutions, which play a fundamental role in the study of the ZK dynamics, are of the form 
	\[ u(t,\bx)=Q_c(x - ct ,\by) \quad  \text{with} \quad Q_c(\bx) \underset{|\bx| \to +\infty}{\longrightarrow} 0 , \]
	where 
	\begin{align}\label{defi:scaling}
		Q_c(\bx) := c Q ( c^{\frac12} \bx ), \quad \text{for} \ c>0, 
	\end{align}
	and $Q$ is the \emph{ground state} of 
	\begin{align}\label{eq:ground_state}
		- \Delta Q+Q -Q^2=0.
	\end{align}
	Indeed, it is well-known that the elliptic equation \eqref{eq:ground_state} has a unique positive solution $Q$, called \emph{ground state}, which is smooth and radially symmetric. We refer to \cite{Strau77,BLP81} for the existence and \cite{Kwo89} for the uniqueness. Note that, in contrast to the one dimensional case, $Q$ is not explicit in dimension $d \ge 2$. However, it was proved in \cite{GNN81} that $Q$ and its derivative decay exponentially at infinity, \textit{i.e.} for any $\alpha \in \mathbb N^d$, there exists $C=C(\alpha)>0$ such that
	\begin{align}\label{asym:Q}
		\forall \vert \bx \vert >1, \quad \left\vert \partial^{\alpha}Q(\bx) \right\vert \leq C \vert \bx \vert^{-\frac{d-1}2} e^{-\vert \bx \vert}.
	\end{align}
	
	The solitary waves $Q_c$ were proved to be orbitally stable in dimension $d=2$ and $d=3$ by de Bouard \cite{deB96}. More recently, the asymptotic stability of the family $\{Q_c\}_{c>0}$ has been proved by C\^ote, Mu\~noz, Pilod and Simpson \cite{CMPS16} in dimension $d=2$ by following the strategy introduced by Martel and Merle for the generalized KdV equation in dimension $1$ (see \cite{MM01,MM05,MM08}). The proof relies on a Liouville property classifying the solutions localized around a solitary wave. A solution close to a solitary wave converges up to a subsequence to a limit object, whose emanating solution satisfies good decay properties. Thanks to the rigidity result, this solution has to be a solitary wave. The decay of the solution is shown by using monotonicity results for portions of the mass and the energy away from the solitary wave, exploiting the fact that solitary waves propagate to the right while the radiation is ejected to the left. 
	
	In higher dimensions, the geometry of the problem is more involved. Monotonicity properties on oblique half-planes were introduced in \cite{CMPS16} in order to reduce the convergence region to a compact set, where a classical convergence result can be used. 
	Another challenge that arises in higher dimensions is the non-explicit nature of the solitary waves. This creates additional problems in proving the rigidity result close to the solitary wave. In particular the virial estimate in \cite{CMPS16}, which is shown through a dual problem introduced by Martel \cite{MA06} in the one dimensional case, relies on the negativity of a scalar product which is checked numerically in dimension $d=2$. However, this sign condition breaks down in dimension $d=3$. To bypass this difficulty, Farah, Holmer, Roudenko and Yang derived in \cite{FHRY23} a new virial estimate in dimension $d=3$ based on different orthogonality conditions and relying on the numerical analysis of the spectra of a linear operator. This allowed them to extend the asymptotic stability result of the ZK solitary waves to the dimension $d=3$. 
 
 Finally, let us also mention the work of Mendez, Mu\~noz, Poblete and Pozo \cite{MMPP21} where new virial estimates in dimension $2$ and $3$ are used to prove the decay of the solutions in large time dependent spatial regions.
	
	\subsection{Finite sum of solitary waves}
	
	In dimension $1$, there exist explicit families of solitary waves called $N$-solitons. These solutions behave asymptotically at time $-\infty$ and $+\infty$ as the  superposition of $N$ well-ordered solitary waves with different speeds (see \cite{Miura76}). Since the solitary waves return to their original shape after interacting, the collision is called \emph{elastic}. The elasticity of the collision is somehow related to the integrable nature of the KdV equation.
	
	In higher dimensions, such explicit families of $N$-solitons are not known. However, the study of solutions behaving at infinity as the sum of $N$ solitary waves is important, since these solutions are expected to be universal attractors of the flow of \eqref{ZK}. This is related to the \emph{soliton resolution conjecture} for ZK, (see \cite{KRS21} for numerical simulations in dimension $d=3$). The existence and uniqueness of such asymptotic $N$-solitary waves was shown by Valet in the $2$ and $3$ dimensional setting \cite{Va21}, by using the strategy introduced by Martel for the gKdV equation in the $1$ dimensional case \cite{Ma05}. 
	
	Another important question concerning the finite sum of solitary waves is that of their stability. In \cite{CMPS16}, the orbital stability of a sum of $N$ well ordered solitary waves was proved in $H^1$ by extending the fundamental work of Martel, Merle and Tsai \cite{MMT02} for the gKdV equation to the $2$-dimensional case. In this paper, we revisit the stability result of \cite{CMPS16}. We first derive a qualitative version of the orbital stability, which will be useful for the study of the collision of solitary waves for ZK in a forthcoming work \cite{PV23}. Then, we also prove the asymptotic stability of a finite sum of well ordered solitary waves. 
	
	We state our main results in the theorem below. Without loss of generality, we will work with two solitary waves.

	\begin{theo}\label{theo:orbital} Let $d=2$ or $d=3$ and
		let $0<\underline{c}<\bar{c}$ be fixed. There exist positive constants $k=k(\underline{c},\bar{c})$, $K=K(\underline{c},\bar{c})$ and $A=A(\underline{c},\bar{c})$ such that the following is true. For two velocities $c_1^0$ and $c_2^0$ satisfying $\underline{c}<c_2^0<c_1^0<\bar{c}$ and $\sigma=c_1^0-c_2^0<\min\{1, \underline{c}\}$, we define
		\begin{align*}
			\alpha^\star:= k\sigma \quad \text{and} \quad Z^\star := K | \ln \sigma| .
		\end{align*}
		Let $\alpha<\alpha^\star$, $Z> Z^\star$, $(z_i^0, \omega_i^0) \in \mathbb{R} \times \mathbb{R}^{d-1}$, $i=1,2$, and $u_0 \in H^1(\mathbb{R}^d)$ satisfy
		\begin{align}\label{eq:IC}
			\left\| u_0 - \sum_{i=1}^2 Q_{c_i^0} \left( \cdot - \left( z_i^0, \omega_i^0 \right) \right) \right\|_{H^1} < \alpha \quad \text{and} \quad z_1^0-z_2^0 >Z,
		\end{align}
		and let $u \in C(\mathbb R: H^1(\mathbb R^d))$ be the solution of \eqref{ZK} evolving from the initial condition $u(0)=u_0$. Then, there exist functions $(z_i, \omega_i, c_i) \in \mathcal{C}^{1}(\mathbb{R}_+ : \mathbb{R}^d \times (0,+\infty))$, $i=1,2,$ such that: 
		\begin{itemize}
			\item[(i)] (Orbital stability) for all $t \ge 0$,
			\begin{align}
				&\left\| u(t) - \sum_{i=1}^2 Q_{c_i(t)} \left( \cdot - (z_i(t), \omega_i(t)) \right) \right\|_{H^1} \leq A \left( \alpha + e^{-\frac{1}{32} \sqrt{\underline{c}}Z }\right); \label{eq:u_orbital} \\
				&\sum_{i=1}^2 \left( \left\vert \dot{z}_i(t) -c_i(t) \right\vert + \left\vert \dot{\omega}_i(t) \right\vert \right) \leq    A \left( \alpha + e^{-\frac{1}{32} \sqrt{\underline{c}}Z }\right)  ; \label{eq:bound_dot:z_omega_i_t}\\
               & z(t)=z_1(t)-z_2(t)\geq \frac12(Z+ \sigma t) ;\label{defi:z} \\
                & \sum_{i=1}^2 \left\vert c_i(t) - c_i^0 \right\vert \leq  A \alpha .\label{eq:bound_c_i_t}
			\end{align}

			\item[(ii)] (Asymptotic stability) the limits $c_1^+=\lim_{+\infty} c_1(t)$, $c_2^+=\lim_{+\infty}c_2(t)$ exist, and  
			\begin{align}
               & \sum_{i=1}^2\left\vert c_i^+ - c_i^0 \right\vert \leq A \left( \alpha + e^{-\frac18 \sqrt{\underline{c}}Z } \right); \\
				&\lim_{t\rightarrow +\infty} \left\| u(t) - \sum_{i=1}^2 Q_{c_i(t)} \left( \cdot - \left( z_i(t), \omega_i(t) \right) \right) \right\|_{H^1(x>\frac{1}{100}\underline{c} t)} = 0 ; \label{conv_asymp}
                \\ & (\dot{z}_i(t),\dot{\omega}_i(t)) \underset{t \rightarrow + \infty}{\longrightarrow} (c_i^+,0), \quad i=1,2 . \nonumber
			\end{align}
		\end{itemize}
	\end{theo}
	
	\begin{rema}
		The theorem can be generalized to $N$ solitary waves with velocities $0<\underline{c}< c_N^0 < \cdots < c_1^0< \bar{c}$ by defining
		\begin{align*}
			\sigma := \min \left(c_{N-1}^0 - c_N^0, \cdots, c_1^0-c_2^0 \right).
		\end{align*}
	\end{rema} 
	
	\begin{rema} Regarding the orbital stability results in \cite{MMT02,CMPS16}, we precise the dependence of the constants $\alpha^\star$ and $Z^\star$ with respect to the difference $\sigma$ between the velocities of the solitary waves. This refinement will be useful to study the dynamics of the collision of nearly equals solitary waves of the ZK equation in a forthcoming work \cite{PV23}.  
		
		For this reason and for the sake of simplicity, we choose to focus on the case where $\sigma=c_1^0-c_2^0 < \min\{1,\underline{c} \}$. Without this restriction, a similar result would hold by redefining $\alpha^\star=k \min\{1,\sigma\}$ and $Z^\star=K\left(1+|\ln(\min\{1,\sigma\})|\right)$.
		
		Finally, we refer to Proposition 3.2 in \cite{MM11} for a different proof of a similar result in one dimension based on a refined energy functional. 
	\end{rema}

 \begin{rema}
It follows from the proof that the asymptotic stability result also holds on half-spaces of the form 
\begin{equation*}
\mathcal{AS}(t,\theta)=\left\{ \bx=(x,\by) \in \mathbb R^d : x-\frac1{100}\underline{c}t+(\tan{\theta})y_j>0 \right\},
\end{equation*}
where $|\theta|<\frac{\pi}3$ and $j=1,2$ ($y_j=y$ if $d=2$). We refer to Remark 1.1 in \cite{CMPS16} for a similar property around one solitary wave.

Moreover, in dimension $d=3$, the convergence might be improved on conical regions of the form 
\begin{equation*}
\mathcal{CAS}(t,\theta)=\left\{ \bx=(x,y_1,y_2) \in \mathbb R^3 : x-\frac1{100}\underline{c}t+(\tan{\theta})\sqrt{1+y_1^2+y_2^2}>0 \right\},
\end{equation*}
$|\theta|<\frac{\pi}3$, arguing as in \cite{FHRY23}.
 \end{rema}

\subsection{Outline of the proof}

The proof of the orbital stability follows the road map introduced in \cite{MMT02} in the one dimensional case, and already used in \cite{CMPS16} in the two dimensional case. The main novelty here is that we quantify the range of values of the distance $\alpha$ of the initial datum $u_0$ to the sum of the two solitary waves and the initial distance of their centers $Z$ in the horizontal direction for which the orbital stability holds in terms of a lower bound $\underline{c}$, a higher bound $\bar{c}$ and the difference between the initial velocities $\sigma:=c_1^0-c_2^0$.

The strategy consists in using a bootstrap argument in a tubular neighbourhood of the sum of the two decoupled solitary waves and assuming for the sake of contradiction that the maximal time interval $[0,t^{\star})$ for which the solution stays in this neighbourhood is finite. 

On this finite time interval, by using modulation theory, the geometrical parameters representing the centers $\bz_i(t)$ and the velocities $c_i(t)$ of the waves are adjusted in such a way that the error $\epsilon(t)$ between the solution and the sum of the solitary waves satisfies suitable orthogonality relations. The conservation of the mass and the energy (see \eqref{def:energy} and \eqref{def:mass}) provides a quadratic bound in the $H^1$-norm of the error $\epsilon$ for the variation of $c_1(t) + c_2(t)$ and of $c_1^2(t)+c_2^2(t)$. Furthermore, since the mass around the first solitary wave is almost conserved, a monotonicity argument gives a signed control of the evolution of $c_1(t)$. Then, an Abel resummation argument yields a quadratic bound for the variation of the speeds $c_1(t)$ and $c_2(t)$ in terms of the $H^1$ norm of the error.

Next, by refining and combining the previous bounds with a coercivity argument for the linearized operator around each solitary wave, we obtain an estimate of the error at time $t$ in terms of the initial error and the horizontal distance between the solitary waves. Note that the coercivity is obtained up to some orthogonality relations, which had been previously  adjusted by modulation theory. This allows to strictly improve the $H^1$-bound of the error on the whole time interval $[0,t^{\star})$, which concludes the proof of Theorem \ref{theo:orbital} (i) by contradiction.

\smallskip

The proof of the asymptotic stability of the sum of two solitary waves in Theorem \ref{theo:orbital} (ii) extends the strategy of Martel and Merle for the generalized KdV equation in one dimension (see \cite{MMT02,MM08}) to the two and three dimensional cases. It relies on the arguments used for the asymptotic stability around one solitary wave (see see \cite{MM01,MM05,MM08} in one dimension, \cite{CMPS16} in two dimensions and \cite{FHRY23} in three dimensions). In particular, the rigidity results (Liouville properties) around one solitary wave proved in the two dimensional case in \cite{CMPS16} and in the three dimensional case in \cite{FHRY23} are the key ingredients. 

In the orbital stability setting, the solution $u$ is close to a sum of two well-ordered  modulated solitary waves. We proceed by induction on the solitary waves from right to left to prove the asymptotic stability. 

First, we show that the solution retranslated around the first solitary wave converges in a half-space to a rescaled ground state. By compactness, we extract from the translated solution, a subsequence which converges on a sequence of times tending to $\infty$ to a limit profile $\tilde{u}_{0,1}$, whose emanating solution $\tilde{u}_1$ is orbitally close to the first solitary wave and satisfies a suitable decay condition. Then, we conclude from the Liouville property that $\tilde{u}_1$ has to be a rescaled ground state. In order to prove the decay of the limit profile $\tilde{u}_1$, we make use of monotonicity properties for portions of the mass and energy of $u$ on oblique half spaces of angle $\vert \theta_0 \vert < \frac{\pi}{3}$ around the hyperplane $\{x=0\}$. This reduces the convergence region to a compact set around the first solitary wave, where classical convergence results can be used. By showing that the limit ground state is independent of the sequence of times, we conclude the proof of the asymptotic stability on the left of the first solitary wave.

Next, we show that the difference between the solution and the first solitary wave, retranslated around the second solitary wave, converges in a half-space to another rescaled ground-state. We follow the same strategy as the one explained above around the first solitary wave and only highlight the main difference. As in the one dimensional case (see \cite{MMT02}), the decay of the solution $\tilde{u}_2$, emanating from the limit profile $\tilde{u}_{0,2}$, is proved by using monotonicity arguments on the mass and the energy from the right of the second solitary wave to the left of the first solitary wave. Note however that in higher dimension, the geometry plays a significant role and one has to carefully adapt the times and angles to ensure that the oblique lines where the derivative of the weight function is localized always stay on the right of the second solitary wave and on the left of the first one. This problem does not arise in the case of one solitary wave \cite{CMPS16,FHRY23} or in the one dimensional case in \cite{MMT02,MM08}.

Finally, we conclude the proof of the asymptotic stability result in a region that goes slower than the two solitary waves using monotonicity properties as previously.

\medskip 
The paper is organized as follows: the first sections are devoted to the case of the dimension $d=2$; in Section \ref{Sec:property}, we introduce several properties which will be used  to prove the quantitative orbital stability in Section \ref{Sec:orbital} and the asymptotic stability in Section \ref{Sec:asympt}. Finally, in Section \ref{Sec:3d_case}, we explain the main modifications of the proof of Theorem \ref{theo:orbital} in dimension $d=3$.

	\subsection{Notation}
	
	\begin{itemize}
		\item The inequality $a \lesssim b$ means that there exists a constant $C$ independent of all the parameters such that $a \leq C b$. This constant may change from line to line.
		
		\item The inequality $a \lesssim_{\underline{c},\bar{c}} b$ means that there exists a function $C=C_{\underline{c}, \bar{c}}$ that depends only on $\underline{c}$ and $\bar{c}$ such that $a \leq C_{\underline{c}, \bar{c}} b$. Once again, the constant may change from line to line. 
		
		\item In the case $d=2$, we will denote $\bx =(x,y) \in \mathbb R^2$. Then, $\vert \bx \vert=(x^2+y^2)^{\frac12}$ and $\langle \bx \rangle=(1+\vert \bx \vert^2)^{\frac12}.$ We specify the notation for the $3$-dimensional case in Section \ref{Sec:3d_case}.
		
		\item We will work with real-valued function $f=f(\bx)=f(x,y): \mathbb R^d \to \mathbb R$. For $1 \le p \le +\infty$ and $s \in \mathbb R$, $L^p(\mathbb R^d)$, respectively $W^{s,p}(\mathbb R^d)$, denotes the standard Lebesgue spaces, respectively Sobolev spaces. We also write $H^s(\mathbb R^d)=W^{s,2}(\mathbb R^d)$. 
		For $f,g \in L^2(\mathbb R^d)$ two real-valued functions, we denote 
		by $\langle f,g \rangle =\int_{\mathbb R^d} f(\bx)g(\bx) d\bx$ their scalar product\footnote{Note that there is no risk of confusion with the japanese bracket defined above for an element $\bx \in \mathbb R^d$.}.

		\item For $c>0$, we define the generator of the scaling symmetry $\Lambda_c$ by
		\begin{align}
			& \Lambda_c (f)(\bx) = \Lambda_c f(\bx) := \frac{d}{d\tilde{c}} \left( \tilde{c} f(\sqrt{\tilde{c}} \cdot) \right)_{\vert \tilde{c}=c}(\bx) = \left( \left(1+\frac{1}{2}\bx \cdot \nabla\right)  f\right)\left( \sqrt{c} \bx \right), \label{defi:Lambda}
		\end{align}
		In the case $c=1$, we denote 
		\begin{equation} \label{defi:Lambda_Q}
			\Lambda Q=\Lambda_1 Q
		\end{equation}
	\end{itemize}

	\section{Decomposition and Properties of a solution close to the sum of \texorpdfstring{$2$}{2} solitons} \label{Sec:property}

 In this section, we state and prove several useful results that will be used in the proof Theorem \ref{theo:orbital_stab} (i). For simplicity, we focus in the case of dimension $d=2$, but the same results hold also in dimension $d=3$.
	
	\subsection{Properties of the linearized operator around the ground state}
	We recall that the ground state $Q$ of \eqref{eq:ground_state} is positive, smooth and radially symmetric. Moreover, it enjoys the exponential decay \eqref{asym:Q}. 
	
	The linearized operator $L$ around the ground state $Q$ is defined by
	\begin{align}\label{defi:L}
		L:= -\Delta+1-2Q.
	\end{align}
	The following proposition gathers some classical results on $L$ (see for example \cite{CMPS16} and the references therein).
	
	\begin{propo}\label{prop:L} The operator $L: H^2(\mathbb R^2) \subseteq L^2(\mathbb R^2) \to L^2(\mathbb R^2)$ satisfies the following properties.
		\begin{enumerate}
			\item[(i)] \emph{Essential spectrum:} $L$ is a self-adjoint operator and its essential spectrum is given by $ \sigma_{ess}(L) = [1,+\infty)$. 
			\item[(ii)] \emph{Non-degeneracy of the spectrum:} $ ker(L) = \text{span} \, \{\partial_x Q, \partial_y Q\}$.
			\item[(iii)] \emph{Negative eigenvalue:} $L$ has a unique negative eigenvalue $-\lambda_0<0$, which is simple, and there exists a radially symmetric positive eigenfunction $\chi_0$associated to $-\lambda_0$. 
			\item[(iv)] \emph{Coercivity:} there exists $C>0$ such that, for any $f\in H^2 \subset L^2$ satisfying $\langle f, \partial_xQ \rangle=\langle f, \partial_yQ \rangle=\langle f, Q \rangle =0$:
			\begin{align*}
				\left\langle L f,f \right\rangle \geq C \| f \|_{L^2}^2.
			\end{align*}
			\item[(v)]\emph{Scaling:} 
			$L\Lambda Q=-Q$. Moreover,  $\langle \Lambda Q , Q \rangle=\frac12 \int Q >0$.
		\end{enumerate}
	\end{propo}
	
	\subsection{Modulation close to the sum of decoupled solitary waves}
	
	We consider a solution $v$ close to a sum of two decoupled and \emph{well prepared} solitary waves. In this subsection, we adjust the set of geometric parameters of translation and velocities of the solitary waves $\Gamma=\left(z_1,z_2,\omega_1,\omega_2,c_1,c_2\right)$ to ensure some important orthogonality relations for the difference between the solution and the sum of the two solitary waves. 
	
	Let $\underline{c}$ and $\bar{c}$ be two positive numbers, such that $0<\underline{c}<\bar{c}$, representing the lower bound and the upper bound $\bar{c}$ of the set of velocities.
	For $Z>1$, $\alpha>0$ and two velocities $\underline{c}<c_2^0<c_1^0<\bar{c}$, we define the sets of geometric parameters
	\begin{equation}\label{defi:Xi}
		\mathbb R^{4}_Z:= \left\{ (z_1, z_2,\omega_1, \omega_2) \in \mathbb{R}^4  \ : \ z_{1}-z_2 > Z,  \right\}.
	\end{equation}
	and the tubular neighbourhood of the sum of the two well-ordered solitary waves $Q_{c_1^0}$ and $Q_{c_2^0}$ by
	\begin{align} \label{def:tub}
		\mathcal{U}_{c_1^0, c_2^0, Z, \alpha} := \left\{ u \in H^1 (\mathbb{R}^2) : \ \inf_{(z_1,z_2,\omega_1,\omega_2) \in \mathbb R^4_Z} \left\| u - \sum_{i=1}^2 Q_{c_i^0} \left( \cdot - (z_i,\omega_i) \right) \right\|_{H^1} < \alpha \right\}.
	\end{align}

	\begin{propo}[Choice of modulation parameters]\label{propo:choice_mod_param}
		Let $0<\underline{c}<\bar{c}$. There exist positive constants $A_1=A_1(\underline{c},\bar{c})$, $\alpha_1^\star=\alpha_1^\star(\underline{c},\bar{c})$, and $Z_1^\star =Z_1^\star(\underline{c},\bar{c})$ such that the following is true. For $0<\alpha< \alpha_1^\star$, $Z>Z_1^\star$ and $\underline{c}<c_2^0<c_1^0<\bar{c}$, let $\mathcal{I}$ be a time interval and $u\in \mathcal{C}(\mathcal{I} : H^1(\mathbb{R}^2))$ be a solution of \eqref{ZK} satisfying $u(t) \in \mathcal{U}_{c_1^0, c_2^0, Z, \alpha}$, for all $t \in \mathcal{I}$.
		Then there exists a unique function $\Gamma=\left(z_1,z_2,\omega_1,\omega_2,c_1,c_2\right) \in \mathcal{C}^1\left(\mathcal{I} : \mathbb{R}^4 \times (0,+\infty)^2\right)$ such that, by defining
		\begin{align}\label{defi:R_n}
			R_i(t,\bx) := Q_{c_i(t)} \left( \bx - \left( z_i(t), \omega_i(t)\right) \right), \ i=1,2, \quad \text{and} \quad \epsilon(t) :=  u(t) - \sum_{i=1}^2 R_i(t) ,
		\end{align}
		we have, for any $t \in \mathcal{I}$ and $i\in \{ 1,2\}$,
		\begin{align}
			& \int R_i(t) \epsilon(t) = \int \partial_x  R_i(t) \epsilon(t) = \int \partial_y  R_i(t)\epsilon(t) = 0, \label{eps:ortho} \\
			& \left\| \epsilon (t) \right\|_{H^1} \leq A_1 \alpha, \label{estimate_modulation:eps} \\ & z(t)= z_1(t)- z_2(t) > Z - A_1 \alpha>\frac{3}4Z, \label{estimate_modulation:z} \\
			&|c_i(t)-c_i^0| \le A_1 \alpha, \label{estimate_modulation:c} \\
            &  \frac{14}{15}\underline{c} < c_2(t) \quad \text{and} \quad c_1(t) < \frac{16}{15}\bar{c}. \label{estimate_modulation:c:bis}
		\end{align}
	\end{propo}
	
	\begin{toexclude}
    \begin{rema}\label{rema:IFT}
			If we assume moreover that, for a given time $t_0 \in I$ and for a given $\Gamma_0=(z_1^0,z_2^0, \omega_1^0, \omega_2^0, c_1^0, c_2^0) \in \Xi_{\underline{c},\bar{c},Z}$, $u$ satisfies the condition
			\begin{align} \label{tube:t0}
				\left\| u(t_0) - \sum_{i=1}^2 Q_{c_i^0} \left( \cdot -(z_i^0, \omega_i^0) \right) \right\|_{H^1} < \alpha,
			\end{align}
			then in addition to \eqref{eps:ortho},\eqref{estimate_modulation:eps} and \eqref{estimate_modulation:z}, we also have, for $i=1,2$, 
			\begin{align*}
				\left\vert z_i(t_0) -z_i^0\right\vert \leq K_1 \alpha,\quad \left\vert \omega_i(t_0)-\omega_i^0 \right\vert \leq K_1\alpha \quad \text{and} \quad \left\vert c_i(t_0) - c_i^0 \right\vert \leq K_1 \alpha.
			\end{align*}
		\end{rema}
	\end{toexclude}
	
	\begin{proof}
		The proof of the existence, uniqueness and continuity of $\Gamma$ relies on classical arguments based on a qualitative version of the implicit function theorem for functions independent of time (see for example Lemma 1 and Lemma 8 in \cite{MMT02}, Proposition 2.2 in \cite{Eyc22} or \cite{PV23} and the references therein).
		
		More precisely, for any $v \in \mathcal{U}_{c_1^0, c_2^0, Z, \alpha}$ and $\Gamma=(z_1,z_2,\omega_1,\omega_2,c_1,c_2) \in \mathbb R^4_Z \times I(c_1^0;\gamma) \times I(c_2^0;\gamma)$, where $I(c_j^0,\gamma)=(c_j^0-\gamma,c_j^0+\gamma)$ and $\gamma>0$ is small enough, we define, for $i=1,2$,
		\begin{align*}
			\bar{R}_i(\bx):= Q_{c_i}(\bx-(z_i,\omega_i)), \quad \bar{\epsilon} (\bx) := v(\bx) - \sum_{i=1}^2 \bar{R}_i(\bx)
		\end{align*}
		and  $\Phi : \mathcal{U}_{c_1^0, c_2^0, Z, \alpha}  \times \mathbb R^4_Z \times I(c_1^0;\gamma) \times I(c_2^0;\gamma)$ by
		\begin{equation*}
			\Phi (v,\Gamma)  = \left( \int \partial_x \bar{R}_1 \bar{\epsilon}, \ \int \partial_x \bar{R}_2 \bar{\epsilon}, \ \int \partial_y \bar{R}_1 \bar{\epsilon}, \ \int \partial_y \bar{R}_2 \bar{\epsilon}, \ \int \bar{R}_1 \bar{\epsilon}, \ \int \bar{R}_2 \bar{\epsilon} \right) . 
		\end{equation*}
		The main point consists in proving  that $d_\Gamma \Phi(v,\Gamma)$ is invertible. 
		This is ensured by decomposing $d_\Gamma \Phi(v,\Gamma)=D+A$, where 
		\begin{equation*}
			D=\text{diag}\left(c_1^2\int (\partial_xQ)^2, c_2^2\int (\partial_xQ)^2, c_1^2\int (\partial_yQ)^2, c_2^2\int (\partial_yQ)^2,\int Q \Lambda Q,\int Q \Lambda Q  \right)
		\end{equation*}
		is invertible (see Proposition \ref{prop:L} (v) ) and, by choosing $\gamma \lesssim \alpha$, $\alpha_1^\star$
		small enough and $Z_1^\star$ large enough depending only on $\underline{c}$ and $\bar{c}$, $\|A\|_{L^{\infty}} \le e^{-\frac12 
			\sqrt{\underline{c}} Z^\star}+\alpha^{\star}<\|D^{-1}\|_{L^{\infty}}^{-1}$ (see the computations in the proof of Lemma \ref{est:R1R2} in Appendix \ref{app:R1R2}). Thus, it follows from the implicit function theorem that there exist a positive constant $A_1=A_1(\underline{c},\bar{c})$ and a unique $C^1$-function $\Gamma:\mathcal{U}_{c_1^0, c_2^0, Z, \alpha}  \to \mathbb R^4_{Z-A_1\alpha} \times I(c_1^0;A_1 \alpha) \times I(c_2^0;A_1 \alpha)$ such that $\Phi(v,\Gamma(v))=0$, for all $v \in \mathcal{U}_{\underline{c}, \bar{c}, Z, \alpha}$.
		
		Now, let $u \in C(\mathcal{I}:\mathbb R^2)$ be a solution of \eqref{ZK} such that $u(t) \in \mathcal{U}_{c_1^0, c_2^0, Z, \alpha}$ for all $t \in \mathcal{I}$. By abuse of notation, we define $\Gamma(t):=\Gamma(u(t))$, so that $\Gamma \in C(\mathcal{I})$. The fact that $\Gamma$ is of class $C^1$ follows from the Cauchy-Lipschitz theorem. 
		
		The orthogonality conditions \eqref{eps:ortho} are equivalent to $\Phi(u(t),\Gamma(t))=0$, $t \in \mathcal{I}$, while estimates \eqref{estimate_modulation:eps}-\eqref{estimate_modulation:c:bis} follow from the fact that $(u(t),\Gamma(t)) \in \mathcal{U}_{c_1^0, c_2^0, Z, \alpha}  \to \mathbb R^4_Z \times I(c_1^0;A_1 \alpha) \times I(c_2^0;A_1 \alpha)$ and the triangle inequality.
	\end{proof}

	Below, we derive a useful lemma to estimate different integrals involving $R_1$ and $R_2$, whose proof is given in Appendix \ref{app:R1R2}. With the notation \eqref{defi:R_n} in hand, we introduce, for $i=1,2$,
	\begin{equation} \label{def:LambdaRi}
		\Lambda R_i(t,\bx):=\Lambda_{c_i(t)}Q(\bx-(z_i(t),\omega_i(t))),
	\end{equation}
	where $\Lambda_{c_i}$ is defined in \eqref{defi:Lambda}.
	
	\begin{lemm} \label{est:R1R2}
		Under the assumptions of Proposition \ref{propo:choice_mod_param}, we have, for $i=1,2$, 
		\begin{align} 
			& \int R_i^2=c_i \int Q^2, \label{id:Ri}\\  
			& \int (\partial_x R_i)^2=\int (\partial_y R_i)^2=c_i^2 \int (\partial_xQ)^2,\label{id:dRi} \\
			&\int (\Lambda_i R_i)^2=c_i^{-1} \int (\Lambda Q)^2.\label{Lambda:dRi}
		\end{align}
		Moreover, if $z=z_1-z_2$, we also have, for $i,j=1,2$ with $i \neq j$,
		\begin{align} 
			& \left| \int R_i R_j \right| \lesssim  \bar{c}e^{-\frac78  \sqrt{\underline{c}} z}, \label{est:R1R2.1}\\
			& \left| \int R_i \Lambda R_j \right| \lesssim  \bar{c}\underline{c}^{-1}e^{-\frac78 \sqrt{\underline{c}} z} , \label{est:R1R2.2}\\  
			& \left| \int \partial_xR_i R_j \right|+\left| \int \partial_yR_i R_j \right| \lesssim  \bar{c}^{\frac32}e^{-\frac78 \sqrt{\underline{c}} z},
			\label{est:R1R2.3} \\ 
			& \left| \int \partial_xR_i \partial_xR_j \right|+\left| \int \partial_xR_i \partial_yR_j \right|+\left| \int \partial_yR_i \partial_yR_j \right| \lesssim  \bar{c}^{\frac52}e^{-\frac78 \sqrt{\underline{c}} z},
			\label{est:R1R2.4} \\ 
			&\left| \int \partial_xR_i \Lambda R_j \right|+\left| \int \partial_yR_i \Lambda R_j \right| 
			\lesssim \bar{c}^{\frac32}\underline{c}^{-1}e^{-\frac78 \sqrt{\underline{c}} z}, \label{est:R1R2.5} \\ 
			&\left| \int \partial_x(R_iR_j) R_i \right|  
			\lesssim \bar{c}^{\frac52}e^{-\frac{14}{15}\sqrt{\underline{c}} z} \label{est:R1R2.6}\\ 
			&\left| \int \partial_x(R_iR_j) \partial_xR_i \right|  
			\lesssim \bar{c}^{3}e^{-\frac{14}{15}\sqrt{\underline{c}} z} . \label{est:R1R2.7}
		\end{align}
	\end{lemm}

	If a solution $u$ falls within the scope of the previous proposition, then additional assumptions on $\alpha$ and $Z$ provide a dynamical control on the modulation parameters.
	
	\begin{propo}[Evolution of the modulation parameters]\label{propo:evol_mod_param}
		Let $0<\underline{c}<\bar{c}$ be fixed. There exist positive constants $k_2=k_2(\underline{c},\bar{c})$, $K_2=K_2(\underline{c},\bar{c})$ and $A_2=A_2(\underline{c},\bar{c})$  such that the following is true.   For two velocities $c_1^0$ and $c_2^0$ satisfying $\underline{c}<c_2^0<c_1^0<\bar{c}$ and 
        \begin{equation} \label{def:sigma}
        \sigma=c_1^0-c_2^0<\min\{1, \underline{c}\},
        \end{equation}
        we define
		\begin{align} \label{def:k2K2}
			\alpha^\star_2:= k_2\sigma \quad \text{and} \quad Z^\star_2 := K_2 | \ln \sigma| .
		\end{align}
		Let $0<\alpha<\alpha_2^{\star}$, $Z>Z_2^{\star}$, $\mathcal{I}$ be a time interval containing $0$ and $u \in C(\mathcal{I} : H^1(\mathbb R^2))$ be a solution of \eqref{ZK} satisfying $u(t) \in \mathcal{U}_{c_1^0, c_2^0, Z, \alpha}$ for all $t \in \mathcal{I}$. Then, under the notation of Proposition \ref{propo:choice_mod_param}, it holds, for all $t \in \mathcal{I}$, 
		\begin{align}
			&\sum_{i=1}^2 \left( \left\vert \dot{c}_i(t) \right\vert  + \left\vert \dot{z}_i(t) -c_i(t) \right\vert + \left\vert \dot{\omega}_i(t) \right\vert \right) \nonumber\\ & \quad \quad \quad \quad \leq A_2 \left( \sum_{i=1}^2\left( \int e^{- \frac{1}{4}\sqrt{\underline{c}} \left\vert \bx - \left( z_i,\omega_i \right) \right\vert } \epsilon^2(\bx) d\bx \right)^\frac12 + e^{-\frac12\sqrt{\underline{c}}\left( Z +  \sigma t \right) } \right), \label{eq:ODE} \\
			& z(t)=z_1(t)-z_2(t)  > \frac34\left(Z +\sigma t\right), \label{eq:boundz} \\
			&\sum_{i=1}^2\left\vert c_i(t) - c_i^0 \right\vert \leq \frac1{32} \sigma. \label{eq:c_n}
		\end{align}
	\end{propo}

	\begin{proof} 
		First, since $\sigma<1$, we choose $k_2$ and $K_2$ small enough to have $\alpha_2^{\star}<\alpha_1^{\star}$ and $Z_2^{\star}>Z_1^{\star}$, so that we are in the setting of Proposition \ref{propo:choice_mod_param}.

		By inserting \eqref{defi:R_n} in \eqref{ZK} and using \eqref{eq:ground_state}, we derive the equation satisfied by $\epsilon$, \textit{i.e.}
		\begin{align} \label{eq:eps}
			\partial_t \epsilon +\partial_x \Delta \epsilon = -\sum_{i=1}^2 \left( \dot{c}_i \Lambda_i R_i + (c_i- \dot{z}_i)\partial_x R_i - \dot{\omega}_i \partial_y R_i \right) - \partial_x \left( 2R_1R_2+ 2\epsilon(R_1+R_2)+ \epsilon^2 \right). 
		\end{align}
		
		\begin{toexclude}
			Since $Z>Z_2^\star \geq Z_1^\star + K_1 \alpha_1^\star$, Proposition \ref{propo:choice_mod_param} provides $z>Z_2^\star -K_1\alpha_1 > Z_1^\star$. To take the (formal) scalars products of the previous quantities with $R_n$, $\partial_x R_n$ and $\partial_y R_n$, we use the estimates
			\begin{align*}
				& \left\langle \partial_t \epsilon, R_n \right\rangle \lesssim \vert \dot{c}_n \vert c_n^{-1} \left( \int e^{-\frac{1}{2}\sqrt{c_n}\left\vert \bx - (z_n,\omega_n)\right\vert} \epsilon^2 \right)^\frac12, & & \left\langle \partial_t \epsilon, \partial_x R_n \right\rangle \lesssim \left( \vert \dot{c}_n \vert +  \left( \vert \dot{z}_n \vert + \vert \dot{\omega}_n \vert \right) c_n^{\frac32} \right) \left( \int e^{-\frac{1}{2}\sqrt{c_n}\left\vert \bx - (z_n,\omega_n)\right\vert} \epsilon^2 \right)^\frac12, \\
				& \left\langle \partial_x \Delta \epsilon, R_n \right\rangle \lesssim c_n^2\left( \int e^{-\frac{1}{2}\sqrt{c_n}\left\vert \bx - (z_n,\omega_n)\right\vert} \epsilon^2 \right)^\frac12, & & \left\langle \partial_x \Delta \epsilon, \partial_x R_n \right\rangle \lesssim c_n^{\frac52} \left( \int e^{-\frac{1}{2}\sqrt{c_n}\left\vert \bx - (z_n,\omega_n)\right\vert} \epsilon^2 \right)^\frac12, \\
				& \left\langle \partial_x(2 R_n \epsilon), R_n \right\rangle \lesssim c_n^2 \left( \int e^{-\frac{1}{2}\sqrt{c_n} \left\vert \bx - (z_n,\omega_n)\right\vert} \epsilon^2 \right)^\frac12, && \left\langle \partial_x (2 R_n \epsilon), \partial_x R_n \right\rangle \lesssim c_n^{\frac52} \left( \int e^{-\frac{1}{2}\sqrt{c_n}\left\vert \bx - (z_n,\omega_n)\right\vert} \epsilon^2 \right)^\frac12, \\
				& \left\langle \partial_x (2 R_1R_2), R_1 \right\rangle \lesssim c_1^\frac32 c_2 e^{-\sqrt{c_2}z}, && \left\langle \partial_x (2 R_1 R_2), R_2 \right\rangle \lesssim c_1 c_2^\frac32 e^{-\sqrt{c}_2 z},\\
				& \left\langle \partial_x (2 R_1R_2), \partial_x R_1 \right\rangle \lesssim c_1^2 c_2 e^{-\sqrt{c_2}z}, && \left\langle \partial_x (2 R_1 R_2), \partial_x R_2 \right\rangle \lesssim c_1 c_2^2 e^{-\sqrt{c}_2 z},
			\end{align*}
		\end{toexclude}
  
		Now, by taking the time derivative of the orthogonality relations \eqref{eps:ortho}, using the equation \eqref{eq:eps}, estimates \eqref{estimate_modulation:eps}-\eqref{estimate_modulation:z} and Lemma \ref{est:R1R2}, we find that
		\begin{align} \label{est:mod:int}
			\sum_{i=1}^2\left( \left\vert \dot{c}_i \right\vert + \left\vert \dot{z}_i - c_i \right\vert + \left\vert \dot{\omega}_i \right\vert \right)
			& \lesssim_{\underline{c},\bar{c}} \sum_{i=1}^2 \left( \int e^{-\frac{1}{2}\sqrt{c_i}\left\vert \bx - (z_i,\omega_i)\right\vert} \epsilon^2 \right)^\frac12 +  e^{-\frac78 \sqrt{\underline{c}} z} \lesssim_{\underline{c},\bar{c}}  \left( \alpha + e^{-\frac12 \sqrt{\underline{c}} Z} \right). 
		\end{align}
		Therefore, by adjusting the constants $k_2$ and $K_2$ depending only on $\underline{c}$ and $\bar{c}$, and by recalling \eqref{def:k2K2}, we deduce that $\vert \dot{z}_j - c_j \vert \leq \frac1{32} \sigma$, for $j=1,2$. This, combined to \eqref{estimate_modulation:c}, yields
		\begin{align*}
			\left\vert \dot{z}_1- \dot{z}_2 - c_1^0 + c_2^0 \right\vert \leq \left\vert \dot{z}_1- \dot{z}_2 - c_1+c_2 \right\vert + \left\vert c_1-c_2 - c_1^0 + c_2^0 \right\vert  \leq \frac1{16} \sigma.
		\end{align*}
		An integration of the previous inequality from $0$ to $t$ implies
		\begin{align}\label{eq:encad_z}
			\left\vert z(t)-z(0) - \sigma t \right\vert \leq \frac1{16} \sigma t,
		\end{align}
		which, combined with \eqref{estimate_modulation:z}, yields the bound \eqref{eq:boundz}. The proof of \eqref{eq:ODE} is then obtained  by reinjecting \eqref{eq:boundz} into \eqref{est:mod:int}, while the proof of \eqref{eq:c_n} is a direct consequence of \eqref{estimate_modulation:c} and the choice of $\alpha_2^{\star}$. 
	\end{proof}

	\subsection{Evolution of the energy}
	
	Recall the definition of the energy in \eqref{def:energy}.
	With $R_1$ and $R_2$ defined in \eqref{defi:R_n}, we introduce
	\begin{align*}
		R(t, \bx):= R_1(t,\bx) + R_2(t, \bx).
	\end{align*}
	
	\begin{lemm}\label{lemm:energy}
		Let $0<\underline{c}<\bar{c}$. There exists a positive constant $A_3=A_3(\underline{c},\bar{c})$ such that the following holds. In the setting of Propositions \ref{propo:choice_mod_param} and \ref{propo:evol_mod_param}, we have, for any $t \in \mathcal{I}$,
		\begin{align*}
			\left\vert \sum_{i=1}^2 \left( E(R_i(t))- E(R_i(0))\right) + \frac12 \int \left( \vert \nabla \epsilon \vert^2 - 2 R \epsilon^2 \right)(t) \right\vert \leq A_3 \left( \| \epsilon (0) \|_{H^1}^2 + \| \epsilon(t) \|_{H^1}^3 + e^{-\frac12 \sqrt{\underline{c}}Z} \right).
		\end{align*}
	\end{lemm}
	
	\begin{proof}
		Let us develop the energy of the solution at a time $t$. We have
		\begin{align*}
			E(u)= E(R) + \int \nabla R \cdot \nabla \epsilon - \int R^2 \epsilon + \int \left( \frac{\vert \nabla \epsilon \vert^2}{2} - R\epsilon^2 \right) - \int \frac{\epsilon^3}{3}.
		\end{align*}
		By using Lemma \ref{est:R1R2}, we have
		\begin{align*}
			\left\vert E(R) - E(R_1)- E(R_2) \right\vert \lesssim \int \left\vert \nabla R_1 \cdot \nabla R_2 \right\vert + \int \left( R_1^2 R_2 + R_1 R_2^2 \right) \lesssim_{\underline{c}, \bar{c}} e^{-\frac78 \sqrt{\underline{c}} z}.
		\end{align*}
		In a similar way, we obtain
		\begin{align*}
			\left\vert \int R^2 \epsilon - \int R_1^2 \epsilon - \int R_2^2 \epsilon \right\vert \leq \int 2 R_1 R_2 \vert \epsilon \vert \lesssim_{\underline{c}, \bar{c}} e^{-\frac78 \sqrt{\underline{c}}z} \| \epsilon \|_{L^2}.
		\end{align*}
		We deduce combining these estimates with \eqref{eq:ground_state}, \eqref{eps:ortho} and \eqref{eq:boundz} that there exists a constant $A_3=A_3(\underline{c},\bar{c})$ such that
		\begin{align*}
			\left\vert E(u(t))- \sum_{i=1}^2 E(R_i(t)) - \int \left( \frac{\vert \nabla \epsilon \vert^2}{2} - R \epsilon^2 \right)(t) \right\vert \leq A_3 \left( e^{-\frac12 \sqrt{\underline{c}}Z} + \| \epsilon(t) \|_{H^1}^3 \right) .
		\end{align*}
		Therefore, we conclude the proof of the lemma by taking the difference of the previous inequality at times $t$ and $0$ and using the conservation of the energy.
	\end{proof}
	
	\subsection{Almost monotonicity of the mass on the right}\label{sec:mass_right}

 In this part, we consider a solution that is close to the sum of two solitary waves and study the evolution of a half portion of the mass of the solution. To this aim, define the weight function $\psi$ by
 \begin{align}
 \psi(x) := \frac{2}{\pi}\arctan (e^x), \label{defi:psi}
 \end{align}
so that $\psi$ is increasing on $\mathbb R$, $\lim_{x\to -\infty}\psi(x)=0$, $\lim_{x\to -\infty}\psi(x)=1$ and $\psi(-x)=1-\psi(x)$. Note also that, for any $x \in \mathbb R$, 
\begin{equation} \label{prop:psi}
\psi' (x)=\frac1{\pi \cosh (x)} \quad \text{and} \quad |\psi^{(3)}(x)| \le \psi' (x) .
\end{equation}
 
We suppose that a solution of \eqref{ZK} $u\in\mathcal{C}^1(\mathcal{I}:H^1(\mathbb{R}^2))$ satisfies $u(t) \in \mathcal{U}_{c_1^ 0,c_2^ 0,Z,\alpha}$ on an interval $\mathcal{I}=[0,t_0]$. As a consequence of Proposition \ref{propo:choice_mod_param}, $u$ is approximated by a sum of two ground states scaled by $c_i(t)$ and translated by $(z_i(t),\omega_i(t))$.  

We define the average function on the $x$-axis by 
	\begin{align}\label{defi:m}
		m(t) := \frac12 \left( z_1(t) + z_2(t)\right). 
	\end{align}
From the two fixed velocities $c_1^0>c_2^0$ we define, inspired by \cite{MMT02, CMPS16}, the parameter $\gamma$, the scaled weight function $\psi_s$ and the localized mass by
		\begin{align}\label{defi:orbital_stab}
		\gamma:=\frac12 \left( c_1^0 + c_2^0\right), \quad \psi_\gamma(x) := \psi \left( \frac{\sqrt{\gamma}}{2} x \right) \quad \text{and} \quad I(t) := \int u^2(t,\bx) \psi_{\gamma} \left( x - m(t)\right) d\bx. 
		\end{align}
Note from \eqref{prop:psi} that the weight function $\psi_s$ satisfies
		\begin{align}\label{eq:psi3}
			\forall x \in \mathbb{R}, \quad \left\vert \psi_\gamma^{(3)}(x) \right\vert \leq \frac{\gamma}{4} \psi_\gamma'(x).
		\end{align}
  
	In Picture \ref{fig:monot_middle}, the two solitary waves are represented by two disks in the plane. The first solitary wave, represented by a black disk thus on the right, is more concentrated and faster than the second solitary wave, that is represented in grey, and is less concentrated and thus moves slower.
  
  \begin{figure}[h]
      \centering
			\begin{tikzpicture}
				
				\draw[->] (-5,0) -- (-1,0);
				\draw (-1,0) node[right]{$x$};
				\draw (-4,1) node[above]{$y$};
				\draw (-2,1) node[above]{time $0$};
				\draw[->] (-4,-1) -- (-4,1);
				
				\fill[opacity=0.5,black] (-5,0.5) circle(0.3);
				\fill[opacity=0.8,black] (-3.5,-0.5) circle(0.2);
				
				\draw[color=red] (-4.25,-1) -- (-4.25,1);
				\draw[color=red] (-4.25,-1) node[below]{$x=m(0)$};
				
				
				\draw[->] (1,0) -- (5,0);
				\draw[->] (2,-1) -- (2,1);
				\draw (5,0) node[right]{$x$};
				\draw (2,1) node[above]{$y$};
				\draw (5,1) node[above]{time $t_0>0$};
				
				\fill[opacity=0.5,black] (2,0.5) circle(0.3);
				\fill[opacity=0.8,black] (4.5,-0.5) circle(0.2);
				
				\draw[color=red] (3.25,-1) -- (3.25,1);
				\draw[color=red] (3.25,-1) node[below]{$x=m(t_0)$};
			\end{tikzpicture}
      \caption{Scheme of the localization of the two solitary waves and the line $x=m(t)$ at two different times}
      \label{fig:monot_middle}
  \end{figure}
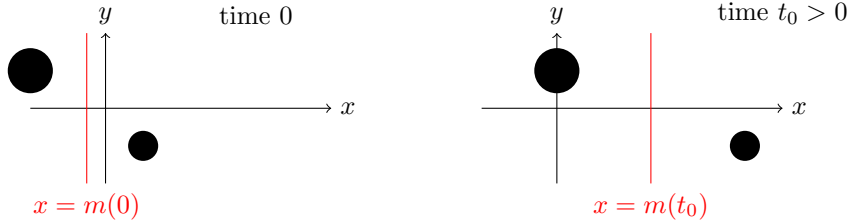

	\begin{lemm}[Almost monotonicity of the mass]\label{lemm:monotonicity_mass}
		There exist positive constants $k_4=k_4(\underline{c},\bar{c})$, $K_4=K_4(\underline{c},\bar{c})$ and $A_4=A_4(\underline{c},\bar{c})$ such that the following is true. Define
		\begin{align}\label{defi:alpha_4}
			\alpha_4^\star:=  k_4\sigma \quad \text{and} \quad 
			Z_4^\star:= K_4 \left\vert \ln \sigma \right\vert ,
		\end{align}
    where we recall that $\sigma=c_1^ 0-c_2^ 0 < \min\{1,\underline{c}\}$. Assume that $Z>Z_4^\star$ and $\alpha< \alpha_4^\star$. Then, we have, for any $t \in \mathcal{I}= [0,t_0]$,
			\begin{align}\label{eq:monotonic_mass}
				I(t)- I(0) \leq A_4  e^{-\frac1{16} \sqrt{\underline{c}}(Z+\sigma t))}.
			\end{align}
	\end{lemm}

 Before giving the proof of Lemma \ref{lemm:monotonicity_mass}, we state some technical estimates for the interactions of the weight function $\psi_{\gamma}$ with the solitary waves $R_i$. 

 \begin{lemm} \label{lemma:est:Rpsi}
 Under the assumptions of Lemma \ref{lemm:monotonicity_mass}, we have
\begin{align} 
				&\| R_1 \psi_\gamma'\left( \cdot - m(t)\right) \|_{L^\infty} + \| R_2 \psi_\gamma'\left( \cdot - m(t)\right) \|_{L^\infty} \lesssim \bar{c} e^{-\frac18 \sqrt{\underline{c}} \left( Z + \sigma t\right)} \label{est:Rpsi} \\ 
				&\left\| R_1 \left(\psi_\gamma \left( \cdot - m\right) -1 \right) \right\|_{L^2} + \left\| R_2 \psi_\gamma\left( \cdot - m\right) \right\|_{L^2} \lesssim \bar{c}^\frac12 e^{- \frac1{16} \sqrt{\underline{c}}\left( Z + \sigma t\right)}. \label{ineq:R_psi_L2}
			\end{align}
 \end{lemm}
 The proof of Lemma \ref{lemma:est:Rpsi} is given in Appendix \ref{app:R1R2}.

	\begin{proof}[Proof of Lemma \ref{lemm:monotonicity_mass}]
		First, we choose $k_4 \le k_2$ and $Z_4 \ge Z_2$ so that we are in the setting of Propositions \ref{propo:choice_mod_param} and \ref{propo:evol_mod_param}. Then, it follows
		from \eqref{estimate_modulation:c} and \eqref{eq:ODE}, that $\gamma < 2\dot{m}(t)$. From this and \eqref{eq:psi3}, we compute
		\begin{align*}
			\frac{d}{dt} I (t)
			& = \int \left( -3 (\partial_x u)^2 - (\partial_y u)^2 + \frac43 u^3 - u^2 \dot{m} \right)  \psi_\gamma'  \left( \cdot - m(t)\right) + \int u^2  \psi_\gamma^ {(3)} \left( \cdot - m(t)\right) \\ 
			& \quad \leq \int \left( -3 (\partial_x u)^2 - (\partial_y u)^2 + \frac{4}{3}u^3 - \frac{\gamma}{4} u^2\right)  \psi_\gamma' \left( \cdot - m(t)\right).
		\end{align*}
		
		We now split the cubic term into
		\begin{align*}
			\int u^3 \psi_\gamma'\left( \cdot - m(t)\right) =  \int R u^2 \psi_\gamma'\left( \cdot - m(t)\right) + \int u^2 (u-R) \psi_\gamma'\left( \cdot - m(t)\right).
		\end{align*}

		From \eqref{est:Rpsi}, the Sobolev embedding $H^1(\mathbb{R}^2) \hookrightarrow L^3(\mathbb{R}^2)$ and \eqref{prop:psi}, we obtain
		\begin{align}
			\int R u^2 \psi_\gamma'\left( \cdot - m(t)\right) & \leq \| u \|_{L^2}^2 \| R \psi_\gamma'\left( \cdot - m(t)\right) \|_{L^\infty} \lesssim \left( \bar{c} + \| \epsilon\|_{L^2}^2 \right) \bar{c}^\frac32 e^{-\frac18 \sqrt{\underline{c}}\left( Z + \sigma t\right)}, \label{eq:mono_nonlinear_1}\\
			\int u^2 (u-R) \psi_\gamma'\left( \cdot - m(t)\right) & \leq \| u \sqrt{\psi_\gamma'\left( \cdot - m(t)\right)} \|_{L^3}^2 \| \epsilon \|_{L^3} \nonumber \\
			& \lesssim \left( \int u^2 \psi_\gamma'\left( \cdot - m(t)\right)  + \int \vert \nabla u \vert^2 \psi_\gamma'\left( \cdot - m(t)\right) \right) \| \epsilon \|_{H^1}. \label{eq:mono_nonlinear_2}
		\end{align}
		Thus, we deduce by using \eqref{estimate_modulation:eps} and by choosing $k_4$ small enough that
		there exists a positive constant $C_{\underline{c},\overline{c}}$ such that
		\begin{align}\label{eq:bound_deriv_mass}
			\frac{d}{dt} I(t) \leq - \frac{\min(1,\gamma)}{8} \int \left( u^2 + \vert \nabla u \vert^2 \right) \psi_\gamma'\left( \cdot - m(t)\right) + C_{\underline{c},\overline{c}} e^{-\frac1{8} \sqrt{\underline{c}}\left( Z + \sigma t\right)}.
		\end{align}
		An integration from $0$ to $t$ concludes the proof of \eqref{eq:monotonic_mass} by taking $K_4$ larger if necessary.
	\end{proof}

\section{Quantified orbital stability} \label{Sec:orbital}
This section is dedicated to the proof of the orbital stability in Theorem \ref{theo:orbital} (i) oin the $2$-dimensional case. 

 \subsection{Bootstrap setting and proof of Theorem \ref{theo:orbital} (i)}
Let  $\underline{c}$, $\bar{c}$ be such that $0< \underline{c}< \bar{c}$. 
For $0<\alpha=\alpha(\underline{c},\bar{c})<\alpha_1^ {\star}(\underline{c},\bar{c})$, $Z=Z(\underline{c},\bar{c})>Z_1^ {\star}(\underline{c},\bar{c})$ and two velocities $c_1^ 0$, $c_2^0$ such that $\underline{c} < c_2^0 <c_1^0 < \bar{c}$ and $\sigma=c_1^0-c_2^0<\min\{1,\underline{c}\}$, let $u_0\in H^1(\mathbb{R}^2)$ and $(z_1^0,z_2^0,\omega_1^0,\omega_2^0) \in \mathbb R^4$ satisfy
		\begin{align} \label{est:u0}
			\left\| u_0 - \sum_{i=1}^2 Q_{c_i^0}\left( \cdot - (z_i^0, \omega_i^0) \right) \right\|_{H^1} < \alpha \quad \text{and} \quad z_1^0-z_2^0>Z.
		\end{align}
 We denote by $u \in C(\mathbb R_+ : H^ 1(\mathbb R^2 ))$ the solution of \eqref{ZK} evolving from $u(0)=u_0$. As long as $u(t) \in \mathcal{U}_{c_1^0,c_2^0,\tilde{Z},\tilde{\alpha}}$ defined in \eqref{def:tub}, for some $0<\tilde{\alpha}<\alpha_1^{\star}$ and $\tilde{Z}>Z_1^{\star}$, we consider the $C^1$ function $\Gamma=(z_1,z_2,\omega_1,\omega_2,c_1,c_2)$ constructed in Proposition \ref{propo:choice_mod_param}. With the notation \eqref{defi:R_n} in hand, we then introduce the bootstrap estimate 
\begin{align} 
&\|\epsilon(t)\|_{H^1} \leq A \left( \alpha + e^{-\frac{1}{32} \sqrt{\underline{c}}Z }  \right) ; \label{BS:eps} 
\end{align}
for some constant $A=A(\underline{c},\bar{c})>1$ to be chosen later, and where $\alpha$ and $Z$ are chosen so that $A \left( \alpha + \sigma^{-\frac12}e^{-\frac{1}{32} \sqrt{\underline{c}}Z }  \right)< \alpha_1^{\star}$ and $\frac12 (Z+\sigma t)>Z_1^{\star}$ (see the statement of Proposition \ref{propo:a_priori}).  

We define 
\begin{equation}
t^{\star}=\sup \left\{ \tilde{t}> 0 : \eqref{BS:eps} \ \text{is satisfied for all} \ t \in [0,\tilde{t}] \right\} .
\end{equation}
Note that by continuity of $u(t)$, there exists $\tau>0$ such that the bootstrap estimate \eqref{BS:eps} holds on $[0,\tau]$, so that $t^{\star}$ is well defined. 

The main goal of the rest of the section is to prove that $t^{\star}=+\infty$. The proof relies on the following key \textit{a priori} estimate. 

\begin{propo}\label{propo:a_priori}
Let  $0< \underline{c}< \bar{c}$ be fixed.
There exist $A(\underline{c}, \bar{c})>1$, $K_5(\underline{c}, \bar{c})>1$ and $k_5(\underline{c},\bar{c})>0$ such that the following is true. For two velocities $c_1^ 0$ and $c_2^0$ such that $\underline{c} < c_2^0 <c_1^0 < \bar{c}$, let $\sigma=c_1^0-c_2^0$. We define
\begin{align}\label{defi:alpha_5}
	\alpha_5^\star := k_5\sigma \quad \text{and} \quad Z_5^{\star}= K_5 |\ln \sigma|,
\end{align}
For any $0<\alpha<\alpha_5^\star$ and $Z>Z_5^{\star}$, if a function $u_0\in H^1(\mathbb{R}^2)$ satisfies \eqref{est:u0}
then, with the above notation in hand, we have, for all $t \in [0,t^{\star}]$,
\begin{align} 
&\|\epsilon(t)\|_{H^1} \leq \frac12 A \left( \alpha + e^{-\frac{1}{32} \sqrt{\underline{c}}Z }  \right) ; \label{a_priori.1} 
\end{align}
\end{propo}

Below we give the proof of Theorem \ref{theo:orbital} (i) by assuming Proposition \ref{propo:a_priori}. 

\begin{proof}[Proof of Theorem \ref{theo:orbital} (i)] We assume for the sake of contradiction that $t^{\star} < \infty$. Then, by Proposition \ref{propo:a_priori}, we see that \eqref{a_priori.1} holds on $[0,t^{\star}]$. Hence, we infer from a continuity argument that there exists $\tau>0$ such that, for all $t \in [0,t^{\star}+\tau],$
\begin{align*} 
&\|\epsilon(t)\|_{H^1} \leq \frac34 A \left( \alpha + e^{-\frac{1}{32} \sqrt{\underline{c}}Z }  \right) . 
\end{align*}
This contradicts the definition of $t^{\star}$. Therefore, we deduce that $t^{\star}=+\infty$, which concludes the proof of \eqref{eq:u_orbital}. 

Observe as explained above that with $A=A(\underline{c},\bar{c})>0$ fixed by Proposition \ref{propo:a_priori}, $k=k(\underline{c},\bar{c})>0$, respectively $K=K(\underline{c},\bar{c})>0$, is chosen small enough, respectively large enough, such that $k \le k_5$, $K>K_5$ ($k_5$ and $K_5$ are given by Proposition \ref{propo:a_priori}). If $\alpha<\alpha^{\star}=k\sigma$, $Z>Z^{\star}=K|\ln \sigma|$, then $A \left( \alpha^{\star} + e^{-\frac{1}{32} \sqrt{\underline{c}}Z^{\star} }  \right)< \min\{\alpha_1^{\star},\alpha_2^{\star}\}$ and $\frac12 (Z^{\star}+\sigma t)>\max\{Z_1^{\star},Z_2^{\star}\}$. In particular, thanks to \eqref{BS:eps}, Propositions \ref{propo:choice_mod_param} and \ref{propo:evol_mod_param} hold on $[0,+\infty)$. Therefore, we conclude from \eqref{estimate_modulation:c}, \eqref{eq:ODE} and \eqref{eq:boundz}, by taking $A$ larger if necessary, that \eqref{eq:bound_dot:z_omega_i_t}, \eqref{defi:z} and \eqref{eq:bound_c_i_t} hold for all $t \ge 0$. This concludes the proof of Theorem \ref{theo:orbital} (i). 
\end{proof}

 \subsection{Proof of Proposition \ref{propo:a_priori}}

Let $A=A(\underline{c},\bar{c})>1$ to be chosen later.  First, we take $k_5=k_5(\underline{c},\bar{c})>0$ small enough and $K_5=K_5(\underline{c},\bar{c})>1$ large enough such that $\alpha_5^{\star}$ and $Z_5^{\star}$ defined in \eqref{defi:alpha_5} satisfy 
\begin{equation} \label{cond:alpha5}
A \left( \alpha_5^{\star} + e^{-\frac{1}{32} \sqrt{\underline{c}}Z_5^{\star} }  \right)< \min\{\alpha_1^{\star},\alpha_2^{\star},\alpha_4^{\star}\} \quad \text{and} \quad \frac12 (Z_5^{\star}+\sigma t)>\max\{Z_1^{\star},Z_2^{\star},Z_4^{\star}\}. 
\end{equation}
Thus the decomposition in Proposition \ref{propo:choice_mod_param} applies, so that the bootstrap estimate \eqref{BS:eps} and the estimates in Proposition \ref{propo:evol_mod_param} and Lemmas \ref{lemm:energy} and \ref{lemm:monotonicity_mass} hold on $[0,t^{\star}]$.

We first prove a bound for the variation of the velocities $c_i(t)$ by a quadratic quantity in $\|\epsilon(t)\|_{H^1}$.

	\begin{lemm}[Quadratic control of the variation of $c_j(t)$]\label{lemm:quadra_control_error}
		There exists a positive constant $A_8=A_8(\underline{c}, \bar{c})$ such that for any $t\in [0, t^\star]$,
		\begin{align} \label{lemm:quadra_control_error.1}
			\sum_{i=1}^2 \left\vert c_i(t) -c_i(0) \right\vert \leq A_8  \left( \| \epsilon(t) \|_{H_1}^2 + \| \epsilon (0) \|_{H^1}^2 + e^{-\frac1{16} \sqrt{\underline{c}} Z}\right).
		\end{align}
	\end{lemm}
	
	\begin{proof}
		
		\textit{First step: first order approximation of the localized mass.} Recall the definition of $I$ in \eqref{defi:orbital_stab}. We claim that there exists $A_6 = A_6(\underline{c}, \bar{c})>0$ such that 
		\begin{align}\label{development_M1}
			\left\vert I(t) - I(0) - \left( c_1(t) -c_1(0) \right) \int Q^2 - \int \epsilon^2(t)\psi_\gamma(\cdot-m(t)) \right\vert \leq A_6 \left( \| \epsilon(0) \|_{L^2}^2 + e^{-\frac1{16} \sqrt{\underline{c}} Z} \right)
		\end{align}
   and
   \begin{align}\label{development_M2}
			\left\vert \left( c_2(t) -c_2(0) + c_1(t) -c_1(0) \right) \int Q^2 + \int \epsilon^2(t) \right\vert \leq A_6 \left( \| \epsilon (0) \|_{L^2}^2 + e^{- \frac18 \sqrt{\underline{c}} Z} \right).
		\end{align}
		
		We expand the localized mass $I=I(t)$ as
		\begin{align}\label{eq:M1_developed}
			I = I_1 + I_2 + I_3 + I_4, 
		\end{align}
		with
		\begin{align*}
			 I_1 &:= \int R_1^2+\int \epsilon^2 \psi_\gamma (x-m), \\  
             I_2 &:= 2\int R\epsilon \psi_\gamma (x-m), \\ 
             I_3 &:= \int R_1^2 \left( \psi_\gamma(x-m)-1\right), \\
			 I_4 &:= \int (R-R_1)(R+R_1) \psi_\gamma(x -m).
		\end{align*}
		
		We now estimate $I_2$, $I_3$ and $I_4$ separately. By using \eqref{ineq:R_psi_L2} and  the orthogonality relation \eqref{eps:ortho}, we have
		\begin{align*}
			\left\vert I_2 \right\vert 
			& \lesssim \left\vert  \int \epsilon R_1 \left( \psi_\gamma(\cdot -m) -1\right) \right\vert + \left\vert  \int \epsilon R_2 \psi_\gamma(\cdot-m) \right\vert  \lesssim \| \epsilon \|_{L^2} \bar{c}^\frac12 e^{-\frac1{16}\sqrt{\underline{c}}\left( Z+ \sigma t\right)}; \\
			\left\vert I_3 \right\vert & \lesssim \| R_1 \|_{L^2} \left\| R_1 \left( \psi_\gamma \left( \cdot -m_1\right)-1 \right) \right\|_{L^2} \lesssim \bar{c}^{\frac12} e^{- \frac1{16} \sqrt{\underline{c}}\left( Z + \sigma t\right)}; \\
			\left\vert I_4 \right\vert &\leq \left\| R_2 \psi_\gamma \left( \cdot -m_1 \right) \right\|_ {L^2} \| 2R_1 + R_2 \|_{L^2} \lesssim \bar{c}^{\frac12} e^{- \frac1{16} \sqrt{\underline{c}}\left( Z + \sigma t\right)}.
		\end{align*}
		Taking the difference of \eqref{eq:M1_developed} between the times $t$ and $0$, using \eqref{id:Ri} and the previous bounds conclude the proof of \eqref{development_M1}. 
   
        Similarly, the proof of \eqref{development_M2} follows from the conservation of the mass between the times $t$ and $0$, \eqref{eps:ortho}, \eqref{id:Ri} and \eqref{est:R1R2.1}. 

\medskip
\noindent \textit{Second step: control of the variation of the quadratic speed.}
 	Since $E(Q_c) = c^2 E(Q)$, it follows from Lemma \ref{lemm:energy} and \eqref{eq:boundz} that 
		\begin{align}\label{eq:third_step}
			\left\vert E(Q) \sum_{i=1}^2 \left( c_i(t)^2 - c_i(0)^2 \right) + \frac12 \int \left( \vert \nabla \epsilon \vert^2 - 2 R \epsilon^2 \right)(t) \right\vert \leq A_3 \left( \| \epsilon (0) \|_{H^1}^2 + \| \epsilon (t) \|_{H^1}^3 + e^{- \frac12 \sqrt{\underline{c}}Z} \right).
		\end{align}

\medskip
\noindent	\textit{Third step: Abel resummation argument.} Let us define the quantity
\begin{equation} \label{defi:D}
D(t)=(c_1(0) - c_2(0)) \left( c_1(t) - c_1(0) \right) + c_2(0) \left( c_1(t) -c_1(0) + c_2(t) - c_2(0)  \right).
\end{equation}
Then, we claim that there exists a positive constant $A_7=A_7(\underline{c}, \bar{c})$ such that
\begin{align}
	\left\vert D(t)  \right\vert - \frac12 \sum_{i=1}^2 \vert c_i(t) - c_i(0) \vert^2 \le A_7 \left( \| \epsilon(0) \|_{H^1}^2 +\| \epsilon (t) \|_{H^1}^2 + e^{-\frac18 \sqrt{\underline{c} }Z} \right) .\label{eq:resummation}
\end{align}

Indeed, by using a resummation argument, we have
			$D(t)=\sum_{i=1}^2 c_i(0) \left( c_i(t) -c_i(0) \right)$,
which combined with the triangle inequality and \eqref{eq:third_step} implies
\begin{align}
	\left\vert D(t) \right\vert
	& = \frac12 \left\vert \sum_{i=1}^2 \left( c_i(t)^2 - c_i(0)^2 \right) -\sum_{i=1}^2(c_i(t) -c_i(0))^2 \right\vert \label{eq:Abel}\\
	& \leq \frac{A_3}{2 E(Q)} \left( \| \epsilon (0) \|_{H^1}^2 + \| \epsilon(t) \|_{H^1}^3 + e^{-\frac12 \sqrt{\underline{c}} Z}\right) +\frac{(1+\bar{c})}{2E(Q)} \| \epsilon (t) \|_{H^1}^2 + \frac12 \sum_{i=1}^2 \vert c_i(t) - c_i(0) \vert^2 \nonumber
\end{align}
Then, we conclude the proof of \eqref{eq:resummation} by using $\| \epsilon (t) \|_{H^1} \leq \alpha_5^\star<1$.

\medskip
\noindent \textit{Fourth step: proof of \eqref{lemm:quadra_control_error.1}.} Note that  $\sigma \leq 2(c_1(0)-c_2(0))$ (see\eqref{eq:c_n}) and $\underline{c} < c_2^0$. Combining the statement that if $a \leq b$ with $b \geq 0$, then $\vert a \vert \leq -a +2b$ with  \eqref{eq:monotonic_mass}  \eqref{development_M1}, \eqref{development_M2} and \eqref{eq:resummation} yields
		\begin{align*}
			\MoveEqLeft
			\frac12 \left(\sigma \left\vert c_1(t) -c_1(0) \right\vert + \underline{c} \left\vert c_1(t) - c_1(0) + c_2(t) - c_2(0) \right\vert \right)\\
			& \leq \left( c_1(0) -c_2(0) \right) \left\vert c_1(t) - c_1(0) \right\vert + c_2(0) \left\vert c_1(t) - c_1(0) + c_2(t) - c_2(0) \right\vert \\
			& \leq  -D(t) + 4 \bar{c} \frac{A_6}{\int Q^2}  \| \epsilon (0) \|_{H^1}^2 + \bar{c} \left( 2 A_4 + 3 \frac{A_6}{\int Q^2} \right) e^{- \frac1{16} \sqrt{\underline{c}} Z} \\
			& \leq \left( 4 \bar{c} \frac{A_6}{\int Q^2} + A_7 \right) \left( \| \epsilon (t) \|_{H^1}^2 + \| \epsilon (0) \|_{H^1}^2 \right) + \left( 2 A_4 + 4 \frac{A_6}{\int Q^2} + A_7 \right) e^{- \frac1{16} \sqrt{\underline{c}} Z} \\
   & \quad + \frac12 \sum_{i=1}^2 \vert c_i(t) - c_i(0) \vert^2.
		\end{align*}
		Thus, we deduce from the previous inequality, \eqref{eq:c_n} and $\sigma <\underline{c}$ that
		\begin{align*}
			\MoveEqLeft
			\frac14 \sigma \left[ \vert c_1(t)-c_1(0) \vert + \vert c_2(t)- c_2(0) \vert \right] \\
			& \leq \frac12 \left( \sigma \left\vert c_1(t) - c_1(0) \right\vert + \underline{c} \left\vert c_1(t)- c_1(0) + c_2(t) -c_2(0) \right\vert \right) \\
			& \leq 2\left( 4 \bar{c} \frac{A_6}{\int Q^2} +A_7 \right) \left( \| \epsilon (t) \|_{H^1}^2 + \| \epsilon (0) \|_{H^1}^2 \right) + 2 \left( 2 A_4 + 4 \frac{A_6}{\int Q^2} + A_7 \right) e^{- \frac1{16} \sqrt{\underline{c}} Z},
		\end{align*}
		which concludes the proof of Lemma \ref{lemm:quadra_control_error}.
	\end{proof}
	
	Next, we give a bound for the error term $\|\epsilon(t)\|_{H^1}$. 
	\begin{lemm}\label{lemm:control_error}
		There exist positive constants $k_9=k_9(\underline{c}, \bar{c})$, $K_9=K_9(\underline{c}, \bar{c})$ and $A_9=A_9(\underline{c}, \bar{c})$ such that if
  \begin{align}\label{defi:alpha_9}
	\alpha_9^\star := k_9\sigma \quad \text{and} \quad Z_9^{\star}= K_9 |\ln \sigma|,
\end{align}
if $0<\alpha<\alpha_9^\star$ and $Z>Z_9^{\star}$, it holds on $[0,t^\star]$,
		\begin{align}\label{eq:control_eps}
			\left\| \epsilon(t) \right\|_{H^1}^2 \leq A_9 \left( \| \epsilon(0) \|_{H^1}^2 +  e^{-\frac1{16} \sqrt{\underline{c}}Z} \right).
		\end{align}
	\end{lemm}

 Before giving the proof of Lemma \ref{lemm:control_error}, we state a coercivity lemma whose proof is obtained arguing as in Lemma 4 \cite{MMT02} with Proposition \ref{prop:L} in hand.
		
		\begin{lemm}[Positivity of the quadratic form \cite{MMT02}]\label{lemm:coercivity}
			Let $\frac{14}{15}\underline{c} \leq c_2(t) \leq c_1(t) \leq \frac{16}{15}\bar{c}$. There exists $\lambda= \lambda (\underline{c})$ and $Z_{10}^{\star}=Z_{10}^{\star}(\underline{c})$ such that, if for any $t\in[0,t^\star]$, $z(t)=z_1(t)-z_2(t)>Z_{10}^{\star}$, then
			\begin{align} \label{lemm:coercivity.1}
				\int  \left( \vert \nabla \epsilon \vert^2 - 2R(t) \epsilon^2(t) + c(t,x) \epsilon^2(t) \right) \geq \lambda \| \epsilon (t) \|_{H^1}^2
			\end{align}
			where $c(t,x):= c_2(t) + (c_1(t)-c_2(t)) \psi_{\gamma}(x-m(t))$.
		\end{lemm}

	\begin{proof}[Proof of Lemma \ref{lemm:control_error}]
		On the one hand, we obtain from the energy bound \eqref{eq:third_step}, the identity \eqref{eq:Abel} and the control of the variation of the velocities \eqref{lemm:quadra_control_error.1} that
		\begin{align}
			\frac12 \int \left( \vert \nabla \epsilon \vert^2 - 2 R \epsilon^2 \right)(t) & \leq - 2 E(Q) D(t) + A_3 \left( \| \epsilon (0) \|_{H^1}^2 + \| \epsilon(t) \|_{H^1}^3 + e^{-\frac12 \sqrt{\underline{c}} Z} \right) \nonumber\\
                & \quad + \left\vert E(Q) \right\vert A_8^2 \left( \| \epsilon(t) \|_{H^1}^2 + \| \epsilon (0) \|_{H^1}^2 +  e^{-\frac1{16} \sqrt{\underline{c}} Z} \right)^2. \label{eq:eps_quadra}
		\end{align}
		On the other hand, by using the identity $-2E(Q)= \frac12\int Q^2= \frac12 M(Q)>0$ (see Appendix B in \cite{CMPS16}), $c_1(0)-c_2(0)>\frac12 \sigma >0$ from \eqref{eq:c_n}, and combining \eqref{development_M1}, \eqref{development_M2} and \eqref{eq:monotonic_mass}, we deduce
		\begin{align}
			- 2 E(Q)D(t) & = \frac12 M(Q) \left(c_1(t)-c_1(0) \right) \left(c_1(0)-c_2(0)\right)+ \frac12 M(Q)\left(c_1(t)-c_1(0)+c_2(t)-c_2(0) \right)  c_2(0) \nonumber \\
			& \leq - \frac12 \int \epsilon^2(t) \left( (c_1(0) -c_2(0))  \psi_{\delta}(\cdot - m(t)) +c_2(0) \right) +  A_4  \bar{c} e^{-\frac1{16} \sqrt{\underline{c}} Z} \nonumber \\
            & \quad + A_6 \bar{c} \left( \| \epsilon(0) \|_{L^2}^2 + e^{-\frac1{16} \sqrt{\underline{c}} Z} \right) \label{eq:eps_quadra2}.
		\end{align}
		Moreover, by Lemma \ref{lemm:quadra_control_error}, it holds
		\begin{align}
			\MoveEqLeft
			\left\vert \int \epsilon^2(t) \left[ \left( \left( c_1(0)-c_2(0)\right) \psi_s \left( \cdot- m_1(t) \right) + c_2(0) \right) - \left( \left( c_1(t)-c_2(t)\right) \psi_s \left( \cdot- m_1(t) \right) + c_2(t) \right) \right] \right\vert \nonumber \\
			& \leq \| \epsilon(t) \|_{L^2}^2 A_8  \left( \| \epsilon(t) \|_{H^1}^2 + \| \epsilon (0) \|_{H^1}^2 + e^{-\frac1{16} \sqrt{\underline{c}} Z}\right).\label{eq:eps_cn_t}
		\end{align}
		Therefore, we deduce by choosing $k_9=k_9(\underline{c},\bar{c})$ small enough and $K_9=K_9(\underline{c},\bar{c})$ large enough, gathering \eqref{lemm:coercivity.1}, \eqref{eq:eps_quadra}, \eqref{eq:eps_quadra2}, \eqref{eq:eps_cn_t} and recalling that $\|\epsilon(t)\|_{H^1}<\alpha_9^{\star}=k_9\sigma$ and $Z>Z_9^{\star}=K_9|\ln \sigma|$   that, for all $t \in [0,t^{\star}]$, 
  \begin{equation*}
\lambda \|\epsilon(t) \|_{H^1}^2 \lesssim_{\underline{c},\bar{c}} \left(\|\epsilon(t)\|_{H^1}^3+\|\epsilon(0)\|_{H^1}^2+e^{-\frac1{16}\sqrt{\underline{c}}Z} \right), 
  \end{equation*}
  which concludes the proof of \eqref{eq:control_eps} by a continuity argument.
	\end{proof}
	
	We now conclude the proof of Proposition \ref{propo:a_priori}. We choose $A(\underline{c},\bar{c})=4\sqrt{A_9(\underline{c},\bar{c})}$ and $k_5=k_5(\underline{c},\bar{c})$, $K_5=K_5(\underline{c},\bar{c})$ such that \eqref{cond:alpha5} holds and $k_5<k_9$, $K_5>K_9$. Then, \eqref{a_priori.1} follows directly from \eqref{eq:control_eps}.  
	
\section{Asymptotic stability} \label{Sec:asympt}

\subsection{Proof of Theorem \ref{theo:orbital} (ii)}

In this section, we prove the asymptotic stability part of Theorem \ref{theo:orbital} (ii) in the $2$-dimensional case. 
	
	Suppose that the initial condition $u_0$ satisfies \eqref{eq:IC}. By using the orbital stability in Theorem \ref{theo:orbital} (i), there exist functions $(z_i,\omega_i,c_i)\in \mathcal{C}^1(\mathbb{R}_+: \mathbb{R}^2 \times(0,\infty))$, $i\in (1,2)$ such that, by recalling that $R_i$ is defined in \eqref{defi:R_n},
	\begin{align} \label{bound:u(t)}
		\left\| u(t) - \sum_{i=1}^2 R_i(t) \right\|_{H^1} \leq \beta.
	\end{align}
	where
	\begin{align}\label{defi:alpha_star}
		\beta = \beta(\alpha, Z) := A \left( \alpha +e^{-\frac{1}{32} \sqrt{\underline{c}} Z}\right)
	\end{align}
 and $\bz_i := (z_i,\omega_i)$. We also denote by $\bz$ the function $\bz = \bz_1 - \bz_2 = (z,\omega)$. Along the proof, $\beta$ can be chosen as small as desired by increasing $K=K(\underline{c},\bar{c})$ and decreasing $k=k(\underline{c},\bar{c})$. 
 
Following \cite{MMT02}, the strategy is to prove inductively that $u$ converges on a half-plane containing each solitary wave. To achieve this goal, two important results will be needed. The first one is the classical orbital stability result around one solitary wave proved in \cite{deB96}. 
	
	\begin{theo}[Orbital stability of a solitary wave \cite{deB96}]\label{theo:orbital_stab}
		Let $0<\underline{c}<\tilde{c}<\bar{c}$. There exist two positive constants $\alpha_{10}=\alpha_{10}(\underline{c},\bar{c})$ and $A_{10}=A_{10}(\underline{c},\bar{c})$ such that if $\tilde{u}_0\in H^1(\mathbb{R}^2)$ satisfies $\left\| \tilde{u}_0 -Q_{\tilde{c}} \right\|_{H^1} \leq \alpha < \alpha_{10}$, the solution $ \tilde{u}$ of \ref{ZK} with $\tilde{u}(0)=\tilde{u}_0$ satisfies
		\begin{align*}
			\sup_{t \in \mathbb{R}} \inf_{\bz \in \mathbb{R}^2} \left\| \tilde{u}(t) - Q_{\tilde{c}} \left( \cdot- \bz \right) \right\|_{H^1} \leq A_{10} \alpha.
		\end{align*}
	\end{theo}

 The second one is a rigidity property classifying the $L^2$-compact solution around a solitary wave, which is obtained by a direct rescaling of Theorem 1.2 in \cite{CMPS16}. 
	
	\begin{theo}[Nonlinear Liouville property around $Q$ \cite{CMPS16}]\label{theo:nonlinear_liouville}
      Let $0<\underline{c}<\tilde{c}<\bar{c}$. There exist two positive constants $\alpha_{11}=\alpha_{11}(\underline{c},\bar{c})$ and $A_{11}=A_{11}(\underline{c},\bar{c})$ such that the following is true. 
		If $0<\alpha<\alpha_{11}$, $\tilde{v} \in \mathcal{C}\left( \mathbb{R}:H^1(\mathbb{R}^2) \right)$ is a solution to \ref{ZK} satisfying for some function $(z_{\tilde{v}}(t),\omega_{\tilde{v}}(t))= \bz_{\tilde{v}}(t)$, 
		\begin{align*}
			\forall t \in \mathbb{R}, \quad \left\| \tilde{v}(\cdot+\bz_{\tilde{v}}(t) ) - Q_{\tilde{c}} \right\|_{H^1(\mathbb{R}^2)} \leq \alpha,
		\end{align*}
		and for any $\delta>0$, there exists $B>0$ such that
		\begin{align} \label{decay:2d}
			\sup_{t \in \mathbb{R}} \int_{\vert x \vert>B} \tilde{v}^2 (t, \bx + \bz_{\tilde{v}} (t)) dx dy  \leq \delta,
		\end{align}
		then there exists $\breve{c}>0$ satisfying $\left|\breve{c}-\tilde{c}\right| \le A_{11} \alpha$ and $(\breve{z},\breve{\omega}) \in \mathbb{R}^2$ such that
		\begin{align*}
			\tilde{v}(t,x,y) = Q_{\breve{c}} \left( x- \breve{c} t - \breve{z}, y - \breve{\omega} \right).
		\end{align*}
	\end{theo}

 We first prove the asymptotic stability on a half-plane containing the first solitary wave. 
 \begin{propo} \label{propo:asymp:1soliton} Under the assumptions of Theorem \ref{theo:orbital}, by taking $k=k(\underline{c},\bar{c})>0$ smaller, $K=K(\underline{c},\bar{c})>0$ and $A=A(\underline{c},\bar{c})>0$ larger if necessary, the limit $c_1^+=\lim_{+\infty}c_1(t)$ exists and
 \begin{align}
 & |c_1^+-c_1^0|<A \left( \alpha +e^{-\frac{1}{32} \sqrt{\underline{c}} Z}\right), \label{propo:asymp:1soliton.1} \\
 & u \left( t, \cdot + \bz_1(t) \right) \underset{t\rightarrow +\infty}{\longrightarrow} Q_{c_1^+} \quad \text{in} \quad H^1(x >-B), \label{propo:asymp:1soliton.3} \\
 &(\dot{z}_1(t),\dot{\omega}_1(t)) \underset{t \rightarrow + \infty}{\longrightarrow} (c_1^+,0), \label{propo:asymp:1soliton.2}
 \end{align} 
 where $(z_1,\omega_1,c_1) \in \mathcal{C}^1(\mathbb R_+ : \mathbb R^2 \times (0,+\infty))$ are the modulation parameters defined in Theorem \ref{theo:orbital} (i).
 \end{propo}
	
The proof of Proposition \ref{propo:asymp:1soliton} relies on the following key result.

	\begin{propo}\label{propo:limit_object_1}
		 Under the assumptions of Theorem \ref{theo:orbital}, by taking $k=k(\underline{c},\bar{c})>0$ smaller and $K=K(\underline{c},\bar{c})>0$ larger if necessary, for any increasing sequence $\{t_n\}_n\rightarrow +\infty$, there exists a subsequence $\{t_{n_k}\}_k$ and $\tilde{u}_{0,1}\in H^1(\mathbb{R}^2)$ such that, for any $B>0$,
		\begin{align*}
			u \left( t_{n_k}, \cdot + \bz_1(t_{n_k}) \right) \underset{k\rightarrow +\infty}{\rightarrow} \tilde{u}_{0,1} \quad \text{in} \quad H^1(x >-B).
		\end{align*}
		
		Moreover, the solution $\tilde{u}_1$ of \ref{ZK} with initial condition $\tilde{u}_1(0)= \tilde{u}_{0,1}$ satisfies for any $t \in \mathbb{R}$
		\begin{align}\label{eq:u_tilde1_bound}
			\left\| \tilde{u}_1 \left(t,\cdot + \tilde{\bz}_1(t) \right) - Q_{c_1^0} \right\|_{H^1} \lesssim \beta
		\end{align}
		with $\beta$ defined in \eqref{defi:alpha_star} and for any $(t,x)\in \mathbb{R}^2$
		\begin{align}\label{eq:u_tilde1_decay}
			\int_y \tilde{u}_1^2 \left( t, \bx + \tilde{\bz}_1(t) \right) dy \lesssim e^{-\tilde{\kappa}_1 \vert x \vert}
		\end{align}
		for some positive constant $\tilde{\kappa}_1$, where $\tilde{\bz}_1=\left( \tilde{z}_1,\tilde{\omega}_1\right)$ corresponds to the function given by the modulation theory around one solitary wave. 
	\end{propo}

	We give the proof of Proposition \ref{propo:asymp:1soliton} assuming that Proposition \ref{propo:limit_object_1} holds. 
 
 \begin{proof}[Proof of Proposition \ref{propo:asymp:1soliton}]
We deduce from Proposition \ref{propo:limit_object_1} and compactness that, for any sequence $\{t_n\}_n\rightarrow +\infty$, there exist a subsequence $\{t_{n_k}\}_k$, a function $\tilde{u}_{0,1} \in H^1(\mathbb R^2)$ and a constant $\tilde{c}_{0,1}>0$ such, that for any $B>0$,
\begin{align} \label{conv.tilde_u_01.1}
	c_1(t_{n_k}) \rightarrow \tilde{c}_{0,1} \quad \text{and} \quad u\left(t_{n_k}, \cdot + \bz_1(t_{n_k}) \right) \rightarrow \tilde{u}_{0,1} \quad \text{in} \quad H^1(x>-B).
\end{align}
Moreover, the solution $\tilde{u}_1$ to \ref{ZK} with initial condition $\tilde{u}_1(0)= \tilde{u}_{1,0}$ satisfies \eqref{eq:u_tilde1_bound} and \eqref{eq:u_tilde1_decay}. 

By applying modulation theory around one solitary wave, (see Lemma 3.1 in \cite{CMPS16}), there exists a unique function $\left( \tilde{\bz}_1,\tilde{c}_1 \right)= (\tilde{z}_1,\tilde{\omega}_1,\tilde{c}_1) \in \mathcal{C}^1(\mathbb{R} : \mathbb{R}^2 \times(0,+\infty))$ such that, by denoting $\tilde{R}_1 (t,\bx) := Q_{\tilde{c}_1(t)} \left( \bx - \tilde{\bz}_1(t)\right)$, it holds for any $t\in \mathbb{R}$
	\begin{align*}
 \begin{cases}
		& \left\| \tilde{u}_1(t) - \tilde{R}_1(t) \right\|_{H^1} \lesssim \beta, \\
		& \int \left( \tilde{u}_1(t) -\tilde{R}_1(t) \right) \tilde{R}_1(t)=\int \tilde{u}_1(t) \partial_x \tilde{R}_1(t)=\int \tilde{u}_1(t) \partial_y \tilde{R}_1(t)=0.
  \end{cases}
	\end{align*}
 Moreover, we infer from the uniqueness of the decomposition around one solitary wave and passing to the limit when $k \to \infty$ in \eqref{eps:ortho} that $\tilde{\bz}_1(0)=0$ and $\tilde{c}_1(0)=\tilde{c}_{0,1}$.
 
	Therefore, we deduce from Theorem \ref{theo:nonlinear_liouville} that there exists $c_1^+$ satisfying $|c_1^+-c_1^0| \lesssim \beta$ and $\bz^+_1=(z_1^+,\omega_1^+)\in \mathbb{R}^2$ such that  
	\begin{align} \label{conv.tilde_u_01.2}
		\tilde{u}_{1}(t,\bx) = Q_{c_1^+} \left( x-c_1^+t-z_1^+,y-\omega_1^+ \right).
    \end{align}
	Notice that the speed $c_1^+$ depends on the sequence $\{t_{n_k}\}_k$. The uniqueness of the decomposition provides that $\bz_1^+=\tilde{\bz}_1(0)=0$ and $c_1^+=\tilde{c}_1(0)=\tilde{c}_{0,1}$. Thus for any $B>0$, the sequence $u\left(t_{n_k},\cdot + \bz_1(t_{n_k}) \right)- Q_{c_1(t_{n_k})}$ goes to $0$ in $H^1(x>-B)$ as $k$ goes to $\infty$. Since it is true for any sequence $\{t_n\}_n$ that tends to $+\infty$, we conclude that, for any $B>0$,
	\begin{align}\label{eq:conv_u-Q_c1}
		u \left( t, \cdot + \bz_1(t) \right) - Q_{c_1(t)} \underset{t \to +\infty}{\longrightarrow} 0 \quad \text{in} \quad H^1(x>-B).
	\end{align}

   We now prove the convergence of the scaling parameter $c_1(t)$ as $t\rightarrow +\infty$. Let us define the function  $\psi_{\underline{c}}(x) := \psi(\frac{\sqrt{\underline{c}}}2 x)$, where $\psi$ is defined in \eqref{defi:psi}. For any $\gamma>0$, there exist a time $T=T(\gamma)>0$ and a distance $x_1=x_1(\gamma)>0$ such that, arguing as in the proof of Lemma \ref{lemm:monotonicity_2}, we have, for any $t>t'>T$ and $x_0>x_1$,
	\begin{align*}
		\int u^2(t,\bx ) \psi_{\underline{c}} (x-z_1(t)+x_0) d\bx \leq \int u^2(t', \bx) \psi_{\underline{c}} (x- z_1(t') + x_0)d\bx + \gamma.
	\end{align*}
		Furthermore, increasing $T$ and $x_1$ if necessary, we get from \eqref{eq:conv_u-Q_c1} and \eqref{asym:Q}, that for any $t>T$,
	\begin{align*}
		\left\vert \int u^2(t, \bx) \psi_{\underline{c}} (x-z_1(t)+x_0) d\bx -\int Q_{c_1(t)}^2(\bx) d \bx \right\vert \leq \gamma.
	\end{align*}
	Gathering the previous inequalities at times $t$ and $t'$ with $t>t'>T$, we obtain
	\begin{align*}
		\int Q_{c_1(t)}^2 \leq \int Q_{c_1(t')}^2 +3 \gamma. 
	\end{align*}
	Since $\int Q_{c}^2= c\int Q^2$, we conclude that there exists $c_1^+>0$ such that
	\begin{align*}
		\left\vert c_1^+ - c_1^0\right\vert \lesssim \beta \quad \text{and} \quad c_1(t) \rightarrow c_1^+.
	\end{align*}

 Finally, the proof of \eqref{propo:asymp:1soliton.2} follows from the modulation estimate around one solitary wave (see (3.5) in \cite{CMPS16}) and \eqref{eq:conv_u-Q_c1}. This concludes the proof of Proposition \ref{propo:asymp:1soliton}. 
 \end{proof}

We continue with the proof of the asymptotic stability on a half-plane containing the two solitary waves. Recall that $\bz$ is defined at the beginning of this section.

\begin{propo} \label{propo:asymp:2soliton} Under the assumptions of Theorem \ref{theo:orbital}, by taking $k=k(\underline{c},\bar{c})>0$ smaller, $K=K(\underline{c},\bar{c})>0$ and $A=A(\underline{c},\bar{c})>0$ larger if necessary, the limit $c_2^+=\lim_{+\infty}c_2(t)$ exists and
 \begin{align}
 & |c_2^+-c_2^0|<A \left( \alpha +e^{-\frac{1}{32} \sqrt{\underline{c}} Z}\right), \label{propo:asymp:2soliton.1} \\
 & \left(u -R_1 \right) (t, \cdot + \bz_2(t))\underset{t\rightarrow +\infty}{\rightarrow} Q_{c_2^+} \quad \text{in} \quad H^1(x >-B), \label{propo:asymp:2soliton.3} \\
 &(\dot{z}_2(t),\dot{\omega}_2(t)) \underset{t \rightarrow + \infty}{\longrightarrow} (c_2^+,0), \label{propo:asymp:2soliton.2}
 \end{align} 
 where $(z_2,\omega_2,c_2) \in \mathcal{C}^1(\mathbb R_+ : \mathbb R^2 \times (0,+\infty))$ are the modulation parameters defined in Theorem \ref{theo:orbital} (i).
 \end{propo}
 
In order to prove Proposition \ref{propo:asymp:2soliton}, we use the following key result.
	
	\begin{propo}\label{propo:limit_object_2}
		Under the assumptions of Theorem \ref{theo:orbital}, by taking $k=k(\underline{c},\bar{c})>0$ smaller and $K=K(\underline{c},\bar{c})>0$ larger if necessary, for any increasing sequence $\{t_n\}_n\rightarrow +\infty$, there exists a subsequence $\{t_{n_k}\}_k$ and $\tilde{u}_{0,1}\in H^1(\mathbb{R}^2)$ such that, for any $B>0$,
		\begin{align}\label{eq:conv_u-Q_c2}
			 \left( u -R_1 \right) \left( t_{n_k}, \cdot+ \bz_2(t_{n_k}) \right) - \tilde{u}_{0,2} \rightarrow 0 \quad \text{in} \quad H^1(x>-B).
		\end{align}
	Moreover, $\tilde{u}_2$ solution to \eqref{ZK} with initial condition $\tilde{u}_2(0)= \tilde{u}_{0,2}$ satisfies for any $t \in \mathbb{R}$
	\begin{align*}
		\left\| \tilde{u}_2 \left( t, \cdot + \tilde{\bz}_2(t) \right) -Q_{c_2^0} \right\|_{H^1} \lesssim \beta
	\end{align*}
with $\beta$ defined in \eqref{defi:alpha_star} and for any $(t,x)\in \mathbb{R}^2$
\begin{align*}
	\int_y \tilde{u}_2^2 \left(t, \bx + \tilde{\bz}_2(t) \right) dy \lesssim e^{-\tilde{\kappa}_2 \vert x \vert}
\end{align*}
for some positive constant $\tilde{\kappa}_2$, where $\tilde{\bz}_2$ corresponds to the function given by modulation theory around one solitary wave.
	\end{propo}

	Assuming that Proposition \ref{propo:limit_object_2} holds, we prove Proposition \ref{propo:asymp:2soliton}.
 
 \begin{proof}[Proof of Proposition \ref{propo:asymp:2soliton}]

 For any increasing sequence $\{t_n\}_n \rightarrow +\infty$, there exist a subsequence $\{t_{n_k}\}_k$, $\tilde{c}_{0,2} >0$ and $\tilde{u}_{0,2} \in H^1(\mathbb R^2)$ such that for any $B>0$, \eqref{eq:conv_u-Q_c2} holds and $\{c_2(t_{n_k})\}_k$ converges to $\tilde{c}_{0,2}$. By using the same arguments as in the proof of Proposition \ref{propo:asymp:1soliton}, we obtain that
	\begin{align}\label{eq:conv_u-Q_c2bis}
	\left( u - R_1 \right) (t, \cdot + \bz_2(t)) - Q_{c_2(t)} \rightarrow 0 \quad \text{in} \quad H^1(x>-B).
	\end{align}
	The convergence of $c_2(t)$ to a speed $c_2^+$ is obtained proceeding as for $c_1^+$. For any $\gamma>0$ there exist $T$ and $x_1$ large enough such that for any $t>t'>T$ and $x_0>x_1$, it holds
	\begin{align*}
		\int u^2(t,\bx) \psi_{\underline{c}} \left( x- z_2(t) +x_0 \right) d \bx \leq \int u^2(t',\bx) \psi_{\underline{c}}(x -z_2(t')+x_0) d\bx + \gamma.
	\end{align*}
 Moreover, by using \eqref{eq:conv_u-Q_c2bis}, we have, by taking $T$ and $x_1$ larger if necessary,
\begin{align*}
	\left\vert \int u^2(t,\bx) \psi_{\underline{c}} \left( x-z_2(t)+x_0 \right) d\bx - \int Q_{c_1(t)}^2 - \int Q^2_{c_2(t)} \right\vert \leq \gamma,
\end{align*}
which yields the existence of the limit $c_2^+$ of $c_2(t)$ satisfying
\begin{align*}
	\left\vert c_2^+ -c_2^0 \right\vert \lesssim \beta.
\end{align*}
 \end{proof}

We are now in position to finish the proof of Theorem \ref{theo:orbital} (ii). 
\begin{proof}[Proof of Theorem \ref{theo:orbital} (ii)] From \eqref{propo:asymp:2soliton.3},
and the convergence of $c_2(t)$ to $c_2^+$, we have the convergence 
\begin{align} \label{conv:left:sol2}
\left( u - R_1 -R_2 \right) (t) \underset{t\rightarrow +\infty}{\rightarrow} 0 \quad \text{in} \quad H^1(x >z_2(t)-B) .
\end{align}
To conclude the proof of Theorem \ref{theo:orbital} (ii), it remains to extend the convergence above to the region $\{x >\frac1{100}\underline{c}t\}$. 

We first prove the following claim
\begin{claim}\label{claim:existence_t'}
	Let $x_0>0$ and define the time
	\begin{align*}
		T=T(x_0):= \inf \left\{ t>0 : \ \frac12 \frac{1}{100} \underline{c} t > z_2(0)-x_0 \ \text{and} \ z_2(0)-x_0+(\frac12 - \frac{1}{100} ) \underline{c}t>0 \right\}.
	\end{align*}
	There exists a constant $\kappa>0$ such that, for any $t>T$, there exists a unique time $t'$ satisfying $ z_2(t') + \frac{1}{200} \underline{c} (t-t') -x_0 = \frac{1}{100} \underline{c} t$.
	Furthermore, it holds $t'>\kappa t$.
\end{claim}

\begin{proof}[Proof of Claim \ref{claim:existence_t'}]
	Consider, for a fixed $t>T$, the function $f(t')= z_2(t')+ \frac12 \frac{1}{100} \underline{c} (t-t') - x_0 - \frac{1}{100} \underline{c} t$ in $\mathcal{C}^1(\mathbb{R})$. It satisfies $\left\vert f'(t')- c_2(t') + \frac12 \frac{1}{100} \underline{c} \right\vert \lesssim \alpha \leq \frac18 \underline{c}$ and thus $ \frac12 \underline{c} < f'(t') < 2 \underline{c}$, which shows that $f$ is an increasing function. 
	
	Moreover, it holds $f(0) =z_2(0)-\frac12 \frac{1}{100} \underline{c} t-x_0<0$. On the other hand, with similar computations to \eqref{eq:encad_z}, it holds $z_2(t)> z_2(0) + \frac12 \underline{c}t$, and by definition of $T$, $f(t)= z_2(t)-x_0 - \frac1{100} \underline{c} t >z_2(0)-x_0+(\frac12 - \frac1{100}) \underline{c}t>0$. By the mean value theorem, there exists a unique time $t'\in (0,t)$ such that $f(t')=0$. Integrating $f'$ from $0$ to $t'$ provides
	\begin{align*}
		0<-z_2(0) + \frac12 \frac1{100} \underline{c} t +x_0 =-f(0) < 2 \underline{c} t',
	\end{align*} 
	which concludes the lower bound on $t'$.
\end{proof}

We introduce the weight function
        \begin{align}\label{defi:phi}
            \phi(x) := \psi \left( \frac{\sqrt{\underline{c}}}{4} x \right) \quad \text{satisfying} \quad \forall x \in \mathbb{R}, \quad \left\vert \phi^{(3)}(x) \right\vert \leq \frac{\underline{c}}{16} \phi'(x),
        \end{align}
where $\psi$ is defined in \eqref{defi:psi}. Let $x_0>0$ and define the function
\begin{align*}
	\check{x}(t)= \check{x}(x_0,t_1,t,\bx) := x-z_2(t_1) - \frac1{200} \underline{c} (t-t_1) +x_0
\end{align*}
and the localized mass
\begin{align*}
	\check{I}_{x_0,t_1}(t) := \int_{\mathbb{R}^2} u(t,\bx)^2 \phi(\check{x}(t)) d\bx.
\end{align*}
Let $t>T(x_0)$ with $T(x_0)$ defined in Claim \ref{claim:existence_t'} and define $t_1$ as the unique time such that $z_2(t)+ \frac1{200} \underline{c} (t-t_1)-x_0=\frac1{100} \underline{c} t$. Following the almost monotonicity in Lemma \ref{lemm:monotonicity_2} away from the two solitary waves, we infer, decreasing $k=k(\underline{c},\bar{c})$ if necessary,
\begin{align*}
	\check{I}_{x_0,t_1}(t)- \check{I}_{x_0,t_1}(t_1) \lesssim e^{-\frac{1}{8}\sqrt{\underline{c}} x_0}.
\end{align*}
Furthermore, recalling the definition of $\epsilon$ in \eqref{defi:R_n} and using the orthogonality relations \eqref{eps:ortho}, it holds for $i\in \{1,2\}$ at time $t$ and $t_1$
\begin{align*}
	\left\vert \int \epsilon(t,\bx)R_i(t,\bx) \phi(\check{x}) d\bx \right\vert = \left\vert \int \epsilon(t,\bx)R_i(t,\bx) \left( 1- \phi(\check{x}) \right) d\bx \right\vert \lesssim e^{-\frac{1}{8} \sqrt{\underline{c}}x_0} .
\end{align*}
Hence, with the decomposition $\epsilon^2 = u^2-2\epsilon(R_1+R_2) - (R_1+R_2)^2$, we obtain
\begin{align*}
	& \int \epsilon^2(t,\bx) \phi(\check{x}(t)) d\bx - \int u^2(t,\bx) \phi(\check{x}(t)) d\bx+ \int \left( R_1(t,\bx) + R_2(t,\bx) \right)^2 \phi(\check{x}(t)) d\bx \lesssim e^{-\frac{1}{8}\sqrt{\underline{c}}x_0}, \\
	& \int \epsilon^2(t,\bx) \phi(\check{x}(t)) d\bx -\int \epsilon^2(t_1,\bx) \phi(\check{x}(t_1)) d\bx \lesssim e^{-\frac{1}{8}\sqrt{\underline{c}}x_0} + \sum_i \left\vert c_i(t)-c_i(t_1) \right\vert +e^{-\frac{1}{32}\sqrt{\underline{c}}(Z+ \sigma t_1)}.
\end{align*}
Passing to the limit as $t$ tends to $+\infty$, recalling that $t_1 \geq \kappa t$, $c_1(t)$ and $c_2(t)$ have finite limits and using the convergence on the left of $R_2$ in \eqref{conv:left:sol2}, we deduce
\begin{align*}
	\limsup_{t\rightarrow + \infty} \int \epsilon^2(t,\bx) \phi(\check{x}(t)) d\bx \lesssim e^{-\frac{1}{8}\sqrt{\underline{c}}x_0}.
\end{align*}
Since the bound holds for any $x_0$ and $\check{x}(t)=x-\frac{1}{100} \underline{c} t$, we obtain
\begin{align*}
	\lim_{t\rightarrow +\infty} \int \epsilon^2(t, \bx) \phi(x-\frac1{100} \underline{c} t) d\bx =0.
\end{align*}
Arguing similarly for the derivatives of $\epsilon$, we infer
\begin{align*}
	\lim_{t\rightarrow +\infty} \int \vert \nabla \epsilon \vert^2 (t, \bx) \phi(x-\frac1{100} \underline{c} t) d\bx =0,
\end{align*}
which concludes \eqref{conv_asymp}.

\end{proof}
	
\subsection{Proof of Proposition \ref{propo:limit_object_1}}

		We recall that the definition of the modulated parameters $(\bz_1,c_1)=(z_1,\omega_1,c_1)$ and $(\bz_2,c_2)=(z_2,\omega_2,c_2)$ in Proposition \ref{propo:evol_mod_param}.
		
		For any increasing sequence $(t_n)_n\rightarrow +\infty$, since the function $u(t, \cdot + \bz_1(t))$ is bounded in $H^1$  and $c_1(t)$ is bounded by $\bar{c}+K_1\alpha$ (see \eqref{estimate_modulation:c}), there exists a subsequence $\{t_{n_k}\}_k$, a function $\tilde{u}_{0,1}$ and a constant $\tilde{c}_{0,1}>0$ such that
		\begin{align}\label{eq:weak_conv_1}
			u\left(t_{n_k}, \cdot + \bz_1(t_{n_k}) \right) \underset{k \rightarrow +\infty}{\rightharpoonup} \tilde{u}_{0,1} \quad \text{in} \quad H^1  \quad \text{and} \quad c_1(t_{n_k}) \rightarrow \tilde{c}_{0,1}.
		\end{align}
		Moreover, since $z(t) \to +\infty$ as $t \to +\infty$ (see \eqref{defi:z}), we have $Q_{c_2(t_{n_k})}(\cdot+\bz(t_{n_k})) \rightharpoonup 0$ in $H^1$ as $k \to \infty$, which combined with \eqref{bound:u(t)} yields
		\begin{align*}
			\left\| \tilde{u}_{0,1} - Q_{\tilde{c}_{0,1}} \right\|_{H^1} 
			& \leq  \liminf_{k \to +\infty} \left\| \left( u -R_1 - R_2 \right) \left( t_{n_k}, \cdot + \bz_1(t_{n_k}) \right) \right\|_{H^1} \leq \beta.
		\end{align*}
		
		Let us denote by $\tilde{u}_1$ the solution to \ref{ZK} with initial condition $\tilde{u}_1(0)=\tilde{u}_{0,1}$. From the global well-posedness result in \cite{Fam} of \ref{ZK} in $H^1$ and the orbital stability result in Theorem \ref{theo:orbital_stab}, $\tilde{u}_1$ is defined in $\mathcal{C}(\mathbb{R}: H^1(\mathbb{R}^2))$ and satisfies
		\begin{align*}
			\left\| \tilde{u}_1 \left( t, \cdot - \tilde{\bz}_1(t) \right) - Q_{\tilde{c}_{0,1}} \right\|_{H^1} \lesssim \beta,
		\end{align*}
		where $(\tilde{z}_1,\tilde{\omega}_1)$ corresponds to the modulation function around one solitary wave, see for instance Lemma 3.1 of \cite{CMPS16}. Since the ideas to prove Proposition \ref{propo:limit_object_1} are similar to the proof of Proposition 4.1 of \cite{CMPS16}, we split the proof of Proposition \ref{propo:limit_object_1} into several steps, summarize each of the steps and underline the differences with \cite{CMPS16}.
		
		\vspace{0.5cm}
  \noindent \textit{First step: Monotonicity property of $u$ on the right.}
To deal with the mass around the first scaled solitary wave, recall that the function $\phi$ is defined in \eqref{defi:phi}.
        
        \begin{lemm}\label{lemm:decay_u_1}
		For any $x_0>0$, we have that
			\begin{align*}
				\limsup_{t\rightarrow + \infty} \int \left( \vert \nabla u \vert^2 + u^2 \right) \left(t, \bx + \bz_1(t) \right) \phi \left( x- x_0 \right) d\bx \lesssim e^{- \frac14 \sqrt{\underline{c}} x_0}.
			\end{align*}
		\end{lemm}
		The proof is based on the monotonicity property of a weighted mass on the right of the first solitary wave at two different times, $t_0$ and $t$ with $t_0>t$, as shown in Figure \ref{fig:decay_u_1}. The detailed proof can be found in Lemma 4.2 in \cite{CMPS16} where the weight function is chosen to be parallel to the vertical axis.
  
		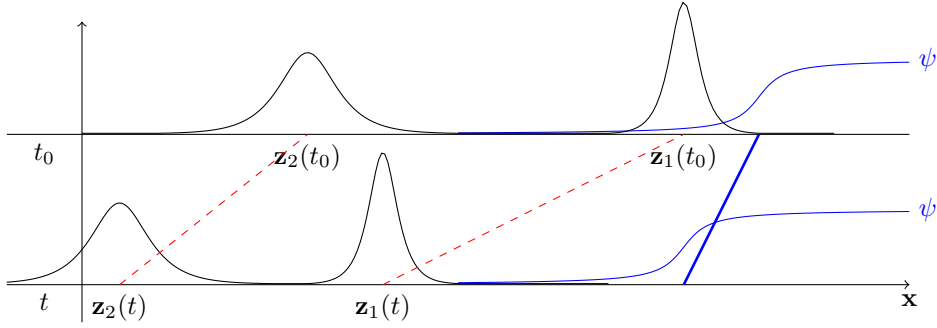
\begin{figure}[ht]
			\centering
			\begin{tikzpicture}[declare function = {
					solwave(\c,\z,\t) =(\c^(0.7))/ (exp(-max(-6,min(\c*(\t-\z),6))) +exp(max(-6,min(\c*(\t-\z),6))));
					weight(\z,\t) = 0.5+atan((\t-\z)*4)/180;}]
				
				\draw [->] (-1,0) -- (11,0);
				\draw (11,0) node[below]{$\bx$};
				\draw [->] (0,-0.5) -- (0,3.5);
				
				\draw (-0.5,0) node[below]{$t$};
				\draw (-1,2) -- (11,2);
				\draw (-0.5,2) node[below]{$t_0$};
				
				\draw [samples=\Num,blue,domain=5:11] plot [variable=\t] ({\t},{weight(8,\t)});
				\draw[blue] (11,1) node[right]{$\psi$};
				\draw [samples=\Num,blue,domain=5:11] plot [variable=\t] ({\t},{2+weight(9,\t)});
				\draw[blue] (11,3) node[right]{$\psi$};
				\draw[blue,line width=1] (8,0) -- (9,2);
				
				\draw [samples=\Num,domain=-1:7] plot [variable=\t] ({\t},{solwave(3,0.5,\t)+solwave(6,4,\t)});
				\draw (0.5,0) node[below]{$\bz_2(t)$};
				\draw (4,0) node[below]{$\bz_1(t)$};
				
				\draw [samples=\Num,domain=0:10] plot [variable=\t] ({\t},{2+solwave(3,3,\t)+solwave(6,8,\t)});
				\draw (3,2) node[below]{$\bz_2(t_0)$};
				\draw (8,2) node[below]{$\bz_1(t_0)$};
				
				\draw[red,dashed] (0.5,0) -- (3,2);
				\draw[red,dashed] (4,0) -- (8,2);
				
			\end{tikzpicture}
		\caption{Evolution of the position of the weight function from the time $t$ to $0$.}
  \label{fig:decay_u_1}
		\end{figure}

  Instead of recalling the proof of Lemma \ref{lemm:decay_u_1}, we introduce the next more general lemma as in Lemma 4.3 in \cite{CMPS16} for which the weight function is based on a non-vertical line. We thus introduce the following notations. Recall that $\phi$ is defined in \eqref{defi:phi}. For $x_0>0$, $t_0>0$, $\vert \theta_0 \vert < \frac{\pi}{3}$, $t \in \mathbb{R}_+$ and $x \in \mathbb{R}$ we define
	\begin{align*}
		\tilde{x}:= \tilde{x}(x_0,\theta_0,t_0,t,\bx)= x-z_1(t_0)-x_0 +\frac{\underline{c}}{2}(t_0-t) + \tan(\theta_0)(y-\omega_1(t_0)) 
   \end{align*}
   and
   \begin{align}
   \tilde{I}_{x_0, t_0, \theta_0} (t) := \int u^2(t,\bx) \phi (\tilde{x}) d\bx, \quad \tilde{J}_{x_0,t_0,\theta_0}(t) := \int \left( \left\vert \nabla u \right\vert^2 (t,\bx) - \frac23 u^3(t,\bx) \right) \phi(\tilde{x}) d\bx. \label{defi:asympt_stab1}
    \end{align}

Notice that $\theta_0$ corresponds to the angle between the vertical axis and the line $\tilde{x}(\bx)=0$, see Figure \ref{fig:evolution_u_right}.

\begin{figure}[h]
    \centering
			\begin{tikzpicture}
				
				\draw[->] (-5,0) -- (-1,0);
				\draw[->] (-4,-1) -- (-4,1);
				\draw (-1,0) node[right]{$x$};
				\draw (-4,1) node[above]{$y$};
				
				\fill[opacity=0.5,black] (-5,0.5) circle(0.3);
				\fill[opacity=0.8,black] (-3.5,-0.5) circle(0.2);
				
				\draw[color=red] (-1,-1) -- (-2.5,1);
				\draw[color=red] (-1,-1) node[below]{$\tilde{x}(t=0)=0$};
    
                \draw[color=blue] (-5,0.5)--(-2.1,0.5);
				\draw[color=blue] (-3.2,0.5) node[above]{$\tilde{d}_2(t_0,0)$};
                \draw[color=blue] (-3.5,-0.5)--(-1.4,-0.5);
				\draw[color=blue] (-2.3,-0.5) node[below]{$\tilde{d}_1(t_0,0)$};
    
				
				\draw[->] (1,0) -- (5,0);
				\draw (5,0) node[right]{$x$};
				\draw (2,1) node[above]{$y$};
				\draw[->] (2,-1) -- (2,1);
				
				\fill[opacity=0.5,black] (2,0.5) circle(0.3);
				\fill[opacity=0.8,black] (4.5,-0.5) circle(0.2);
				
				\draw[color=red] (5.5,-1) -- (4,1);
				\draw[color=red] (5.5,-1) node[below]{$\tilde{x}(t=t_0)=0$};
				\draw[color=red,->] (4.75,0.6) arc (90:135:0.5);
				\draw[color=red] (4.75,0.5) node[above]{$\theta_0$};

                \draw[color=blue] (2,0.5)--(4.4,0.5);
				\draw[color=blue] (3,0.5) node[above]{$\tilde{d}_2(t_0,t_0)$};
                \draw[color=blue] (4.5,-0.5)--(5.1,-0.5);
				\draw[color=blue] (4.5,-0.7) node[left]{$\tilde{d}_1(t_0,t_0)$};
				
			\end{tikzpicture}
    \caption{Evolution of the position of the line $\tilde{x}=0$ at two different times : $t=0$ on the left and $t=t_0>0$ on the right.}
    \label{fig:evolution_u_right}
\end{figure}
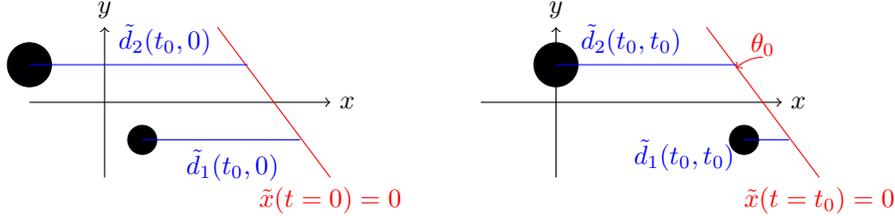
  
  We next compare the weighted mass at two different times, as in \cite{MMT02,MM05,CMPS16}.

		\begin{lemm}\label{lemm:monotonicity_2}
			Let $\vert \theta_0 \vert < \frac{\pi}{3}$, there exists a constant $C=C (\theta_0)$ such that for any $x_0>0$ the following holds. There exists a time $T_0=T_0(\theta_0,x_0)$ such that for any $t_0>T_0$ and $t\in [0,t_0]$, it holds
			\begin{align}\label{eq:monotonic_mass2}
				\tilde{I}_{x_0,t_0,\theta_0}(t_0) - \tilde{I}_{x_0,t_0,\theta_0}(t) \leq C e^{-\frac14 \sqrt{\underline{c}} x_0}, \quad \tilde{J}_{x_0,t_0,\theta_0}(t_0) -\tilde{J}_{x_0,t_0,\theta_0}(t) \leq C e^{-\frac14 \sqrt{\underline{c}}x_0}
			\end{align}
			and 
			\begin{align}\label{eq:limsup_T1}
				\limsup_{t\rightarrow +\infty} \int \left( \vert \nabla u \vert^2 + u^2 \right) \left( t, \bx + \bz_1(t) \right) \phi (x + \tan(\theta_0)y - x_0) d\bx \leq C e^{-\frac14\sqrt{\underline{c}} x_0}.
			\end{align}
		\end{lemm}

	\begin{proof}[Proof of Lemma \ref{lemm:monotonicity_2}]
		We first explain the need of $T_0$. Since $\bz_i(t)$ stands for the center of the $i^{\text{th}}$-solitary wave, we expect that the line $\tilde{x}=0$ at $y=\omega_i(t)$ is at the right of $z_i(t)$ and at a distance larger than $x_0$. We thus define for fixed $\theta_0$ and $x_0$,
  \begin{align*}
      \tilde{d}_1(t_0,t) & = \left( z_1(t_0) + x_0 - \frac12 \underline{c} (t_0-t) - \tan(\theta_0) \left( \omega_1(t) - \omega_1(t_0) \right) \right) - z_1(t)  \\
      \tilde{d}_2(t_0,t) & = \left( z_1(t_0) + x_0 - \frac12 \underline{c} (t_0-t) - \tan(\theta_0) \left( \omega_2(t)- \omega_1(t_0) \right) \right)- z_2(t).
  \end{align*}
  We define the time
		\begin{align*}
			T_0:= \inf \left\{ t_0 \in \mathbb{R}_+ ; \quad \tilde{d}_1(t_0,t_0)\geq x_0 \ \text{and} \ \tilde{d}_2(t_0,t_0) \geq x_0 \right\}.
		\end{align*}
  First observe that $T_0$ is well-defined. Indeed, $\tilde{d}_1(t_0,t_0)=x_0$ and  the function $f(t_0) := \tilde{d}_2(t_0,t_0)$ is increasing on $\mathbb R_+$ and tends to $+\infty$, since $\left\vert f'(t_0)- \sigma \right\vert \leq \frac14 \sigma$.
  
  We claim that for any $t_0>T_0$ and $t \in [0,t_0]$, the distances are larger than $x_0$, \textit{i.e.} $\tilde{d}_1(t_0, t)\geq x_0$ and $\tilde{d}_2(t_0,t) \geq x_0$. Indeed, by \eqref{eq:bound_dot:z_omega_i_t} and \eqref{eq:bound_c_i_t}, for $i\in \{1,2 \}$, it holds $\left\vert \partial_t \tilde{d}_i(t_0,t) - \frac12 \underline{c} + c_i^0 \right\vert \lesssim \alpha \leq \frac14 \underline{c}$, so that  $:t \mapsto \partial_t\tilde{d}_i(t_0,t)$ decreases on $[0,t_0]$. From the definition of $T_0$, we obtain that $\tilde{d}_i(t_0,t)\geq x_0$ for any $t \in [0,t_0]$. 
  
		We continue with the proof of \eqref{eq:monotonic_mass2} following \cite{CMPS16}. We detail the estimate for $\tilde{I}_{x_0,t_0,\theta_0}$ and combine the same reasoning with the one in the proof of Lemma 3.4 in \cite{CMPS16} for the estimate on $\tilde{J}_{x_0,t_0,\theta_0}$. By \eqref{defi:phi}, we have
	\begin{align*}
		\frac{d}{dt} \tilde{I}_{x_0,t_0,\theta_0} (t) 
		& = - \int \left( 3 (\partial_x u) ^2 + ( \partial_y u)^2 +2 \partial_x u \partial_y u \tan(\theta_0) + \frac{\underline{c}}{2} u^2 +  \frac43 u^3  \right) \phi'(\tilde{x}) \\
		& \quad + \int u^2 \left( 1+ \tan^2(\theta_0) \right) \phi^{(3)}(\tilde{x}).
	\end{align*}
	
	Using Young inequality with a coefficient $\kappa_0^2 \in (1, 3 \tan(\theta_0)^{-2})$, with $\vert \theta_0 \vert < \frac{\pi}{3}$, we have
	\begin{align*}
		\left\vert \int 2 \partial_x u \partial_y u \tan(\theta_0) \phi'(\tilde{x}) \right\vert \leq \kappa_0^2 \tan^2(\theta_0) \int (\partial_x u)^2 \phi'(\tilde{x})  + \frac{1}{\kappa_0^2} \int (\partial_y u)^2 \phi'(\tilde{x}) 
	\end{align*}
	and thus for a constant $C= C(\theta_0)$, 
	\begin{align*}
		-\int \left( 3 (\partial_x u )^2 + (\partial_y u)^2 + 2 \tan(\theta_0) \partial_x u \partial_y u \right) \phi'(\tilde{x}) \leq - C \int \vert \nabla u \vert ^2 \phi'(\tilde{x}).
	\end{align*}
	From $\tan(\theta_0)^2 < 3$ and \eqref{defi:phi}, we have
	\begin{align*}
		-\int \frac{\underline{c}}{2} u^2 \phi'(\tilde{x}) + \int u^2 (1+\tan^2(\theta_0)) \phi^{(3)}(\tilde{x}) \leq - \frac{\underline{c}}{4} \int u^2 \phi'(\tilde{x}).
	\end{align*}
	To deal with the non-linear term, we proceed as in \eqref{eq:mono_nonlinear_1}-\eqref{eq:mono_nonlinear_2}. First, notice that an integration of $\partial_t \tilde{d}_i(t_0,t) \leq - \frac14 \underline{c}$ from $t$ to $t_0$ provides  $\tilde{d}_i(t_0,t) \geq x_0 + \frac14 \underline{c} (t_0-t)$. Since $Q$ is rotationally invariant and by performing the change of variable $(x',y')= (1+\tan^2(\theta_0))^{-\frac12} \left( x+ \tan(\theta_0)y, \tan(\theta_0)x -y \right)$, it holds using similar arguments as in the proof of estimate \eqref{est:Rpsi},
	\begin{align}
		\left\| R_i(t) \phi'(\tilde{x}) \right\|_{L^\infty} 
		& = \frac{\sqrt{\underline{c}}}{4} \left\| c_i(t) Q(\sqrt{c_i(t)}\bx) \psi' \left(\frac{\sqrt{\underline{c}}}{4} \left( \sqrt{1+\tan^2(\theta_0)} x + \tilde{d_i}(t_0,t) \right) \right) \right\|_{L^\infty} \nonumber \\
		& \lesssim \exp \left( - \frac14 \sqrt{\underline{c}} \left(x_0 + \frac{\underline{c}}{4} (t_0-t)\right) \right).\label{sol_wave_weight}
	\end{align}
 We obtain the estimate for the non-linear terms
	\begin{align}\label{eq:cubic_term_2}
		\left\vert  \int u^3 \phi'(\tilde{x}) \right\vert \lesssim \| \epsilon \|_{H^1} \left( \int \left( \vert \nabla u \vert^2 + u^2 \right) \phi'(\tilde{x}) \right) + e^{-\frac14 \sqrt{\underline{c}} \left( x_0 + \frac14 \underline{c} (t_0-t)\right)}.
	\end{align}
 
	Gathering the previous inequalities and \eqref{eq:boundz} we get the bound
	\begin{align*}
		\frac{d}{dt} \tilde{I}_{x_0,t_0,\theta_0}(t) \lesssim e^{-\frac14 \sqrt{\underline{c}} \left( x_0 + \frac14 \underline{c} (t_0-t)\right)}.
	\end{align*}
	An integration from $t$ to $t_0$ concludes \eqref{eq:monotonic_mass2}.
	
	The proof of \eqref{eq:limsup_T1} relies on estimate \eqref{eq:monotonic_mass2} and $\lim_{t_0 \rightarrow +\infty} \tilde{I}_{x_0,t_0,\theta_0}(t) =0$.
	\end{proof}
		
		\vspace{0.5cm}
  \noindent \textit{Second step : Strong $L^2$ convergence of $u(t_{n_k}, \cdot + \bz_1(t_{n_k}))$ to $\tilde{u}_{0,1}$ on the right.}
		
		\begin{lemm}\label{lemm:step2}
			It holds, for any $B>0$,
			\begin{align*}
				v_{k,0,1}:= u \left(t_{n_k}, \cdot + \bz_1(t_{n_k}) \right)- \tilde{u}_{0,1} \underset{k\rightarrow +\infty}{\rightarrow} 0 \quad \text{in} \quad L^2(x>-B).
			\end{align*}
		\end{lemm}
		
		The proof of this lemma is the same as the one of Lemma 4.4 in \cite{CMPS16}. It relies on the covering of the half-space ${x>-B}$ by three regions. For any $\varepsilon>0$, the half-plane is covered by the three (overlapping) following regions. By choosing $x_0$ large enough and $\theta_0\in (0,\frac{\pi}{3})$, we obtain the existence of $R=R_{\theta_0,x_0}$ such that by the first step and the exponential decay of $\tilde{u}_{1,0}$,  the mass of $v_{k,1,0}$ on each half-plane $\left\{ x\pm y\geq R \right\} $ is lower than $\frac{\varepsilon}{3}$. On the remaining triangle $\left\{ \bx \in \mathbb{R}^2 : x+y \leq R, x-y \leq R, x >-B \right\}$, the error $v_{k,1,0}$ tends to $0$ as $k\rightarrow +\infty$ by the Rellich-Kondrachov theorem. Choosing $k$ large enough concludes the proof of the lemma.
		
		\vspace{0.5cm}
  \noindent \textit{Third step : Exponential decay of $\tilde{u}_{0,1}$ on the right on finite time intervals.}
		
		\begin{lemm}\label{lemm:step3_1}
		      For any $x_0>0$,
			\begin{align} \label{lemm:step3_1.1}
				\int \left( \vert \nabla \tilde{u}_{0,1} \vert ^2 + \tilde{u}_{0,1}^2 \right) (\bx) \phi (x-x_0) d\bx \lesssim e^{-\frac14 \sqrt{\underline{c}} x_0}.
			\end{align}
			Moreover, for all $t_0\geq 0$, there exists $K(t_0)>0$ such that
			\begin{align} \label{lemm:step3_1.2}
				\sup_{t\in [0,t_0]} \int \left(\vert \nabla \tilde{u}_1\vert^2 + \tilde{u}_1^2 \right) \left(t, \bx \right) e^{\frac14 \sqrt{\underline{c}}x} dx \leq K(t_0).
			\end{align}
		\end{lemm}
		By weak convergence \eqref{eq:weak_conv_1}, it holds for any $x_0>0$
		\begin{align*}
			\left\| \tilde{u}_{0,1} \sqrt{\phi (\cdot-x_0)} \right\|_{H^1} \leq \liminf_{k} \left\| u \left( t_{n_k}, \cdot +  \bz_1(t_{n_k}) \right) \sqrt{\phi (\cdot-x_0)} \right\|_{H^1},
		\end{align*}
		which combined to Lemma \ref{fig:decay_u_1} concludes the proof of \eqref{lemm:step3_1.1}.
		
		The proof of \eqref{lemm:step3_1.2} is similar to the one of Lemma 4.5 in \cite{CMPS16}. The time derivative of the mass and the energy provides an estimate which can be integrated by Gr\"{o}nwall's inequality. In particular, the constant of integration depends on the size of the interval $t_0$.
		
		\vspace{0.5cm}
  \noindent \textit{Fourth step: Strong $L^2$ convergence of $u \left(t_{n_k}+t, \cdot + \bz_1(t_{n_k}) \right)$ to $\tilde{u}_1(t)$ on the right.}
		
		\begin{lemm}\label{lemm:step4.1.1}
			For any $B>0$ and $t \in \mathbb{R}$, it holds
			\begin{align}\label{eq:step4.1.1}
				u \left( t_{n_k}+t, \cdot + \bz_1(t_{n_k}) \right) - \tilde{u}_1(t) \underset{k \rightarrow + \infty}{\rightarrow} 0 \quad \text{in} \quad L^2(x>-B)
			\end{align}
			and for any $t\in \mathbb{R}$
			\begin{align}\label{eq:step4.1.2}
				u \left( t_{n_k}+t, \cdot + \bz_1(t_{n_k}) \right) - \tilde{u}_{1}(t) \underset{k \rightarrow +\infty}{\rightharpoonup} 0 \quad \text{in} \quad H^1(\mathbb{R}^2).
			\end{align}
			Moreover, with $\tilde{\bz}_1=(\tilde{z}_1, \tilde{\omega}_1)$ the $\mathcal{C}^1$-function associated to the modulation parameters of $\tilde{u}_1$, it holds
			\begin{align}\label{eq:step4.1.3}
				\bz_1(t+t_{n_k}) - \bz_1(t_{n_k}) \underset{k \to +\infty}{\rightarrow} \tilde{\bz}_1(t) \quad \text{and} \quad \tilde{\bz}_1(0)=0.
			\end{align}
		\end{lemm}
		
		The proof is similar to the proof of Lemma 4.6 in \cite{CMPS16}. Let us define
  \begin{align}\label{defi:v_k1}
      v_{k,1}(t,\bx) := u\left(t+t_{n_k},\bx+ \bz_1(t_{n_k}) \right) - \tilde{u}_1(t,\bx).
  \end{align}
  Computing the time derivative of the mass of the error yields
  \begin{align}\label{eq:kato_v_k1}
      \frac{d}{dt} \int v_{k,1}^2(t,\bx) \phi(x) d\bx + \int \vert \nabla v_{k,1}(t,\bx) \vert ^2 \phi' (x) d\bx \leq \tilde{K}(t_0) \int v_{k,1}(t,\bx)^2 \phi (x) d\bx. 
  \end{align}
  By a Gronwall's inequality on any interval $[0,t_0]$ and \eqref{lemm:step3_1.2}, we obtain a bound on the weighted mass
		\begin{align*}
			\sup_{t \in [0,t_0]}\int_{\mathbb{R}^2} v_{k,1}(t,\bx)^2 \phi (x) d\bx \leq \tilde{K}(t_0) \int_{\mathbb{R}^2} v_{k,1}(0,\bx)^2 \phi (x) d\bx.
		\end{align*} 
		This inequality combined with Lemma \ref{lemm:step2} provides \eqref{eq:step4.1.1} for $t_0\geq0$, whereas \eqref{eq:step4.1.2} is deduced by uniqueness of the weak limit. The result for the negative interval $[\tilde{t}_1;0]$ is obtained by applying the same arguments to the sequence of functions $ u\left(t_{n_k}+\tilde{t}_1, \cdot + \bz_1(t_{n_k}) \right)$ and conclude by uniqueness of the Cauchy problem.
		
		The convergence result \eqref{eq:step4.1.3} is deduced by passing to the limit in the orthogonality relations and the uniqueness of the modulation parameters.
		
		\vspace{0.5cm}
  \noindent \textit{Fifth step: Exponential decay of $\tilde{u}_{1}$ on the right.}
		
		\begin{lemm}\label{lemm:step5_1}
			For any $t \in \mathbb{R}$, $x_0>0$
			\begin{align}\label{eq:step5.1.1}
				\int \left( \vert \nabla \tilde{u}_1 \vert^2 + \tilde{u}_1^2 \right) \left(t, \bx +\tilde{\bz}_1(t) \right) \phi (x-x_0) d\bx \lesssim e^{-\frac14 \sqrt{\underline{c}} x_0}.
			\end{align}
			Furthermore, for any $t\in \mathbb{R}$ and $x>0$,
			\begin{align}\label{eq:step5.1.2}
				\int_y \tilde{u}_1^2 \left(t, \bx+ \tilde{\bz}_1(t) \right) dy \lesssim e^{-\frac14 \sqrt{\underline{c}} x}.
			\end{align}
		\end{lemm}
		The exponential decay in \eqref{eq:limsup_T1} and the weak convergence in \eqref{eq:step4.1.2} provide \eqref{eq:step5.1.1}. Then, \eqref{eq:step5.1.2} is a consequence of the Sobolev embedding $H^1(\mathbb{R}) \hookrightarrow L^\infty(\mathbb{R})$ and the Cauchy-Schwarz inequality applied to \eqref{eq:step5.1.1} (see the proof of Lemma 4.7 of \cite{CMPS16}).
		
		\vspace{0.5cm}
  \noindent\textit{Sixth step: Strong $H^1$-convergence of $u\left(t_{n_k}+t, \cdot + \bz_1(t_{n_k}) \right)$ to $\tilde{u}_{1}(t)$ on the right.}
		
		\begin{lemm}
			For any $t\in \mathbb{R}$ and $B>0$, it holds
			\begin{align*}
				u \left( t+t_{n_k}, \cdot + \bz_1(t_{n_k}) \right) - \tilde{u}_{1}(t) \underset{n \rightarrow + \infty}{\rightarrow} 0 \quad \text{in} \quad H^1(x>-B).
			\end{align*}
		\end{lemm}
		
		Let us paint the proof with a broad brush, the details may be found in Lemma 4.8 of \cite{CMPS16}. As in the fourth step, it suffices to prove the convergence for a fixed positive time $t_0$. First, an integration of the time derivative of the mass of the error term as in \eqref{eq:kato_v_k1} gives
		\begin{align}\label{eq:step6_small_int}
			\int_{t_0-1}^{t_0} \int \vert \nabla v_{k,1} \vert^2(t, \bx) \phi (x) d\bx \underset{k \rightarrow +\infty}{\rightarrow} 0.
		\end{align}
		Then, computing the time derivative of the energy of the error yields
		\begin{align*}
			\frac{d}{dt} \int \left( \frac12 \vert \nabla v_{k,1} \vert^2 + \frac13 v_{k,1}^3 \right) (t,\bx) \phi(x) d\bx \leq K_2(t_0) \int \left( \vert \nabla v_{k,1} \vert ^2 + v_{k,1} \right)(t,\bx) \phi(x) d\bx.
		\end{align*}
		A first integration from $t \in [t_0-1,t_0]$ to $t_0$ provides
		\begin{align*}
			\MoveEqLeft
			\int \vert \nabla v_{k,1} \vert^2 (t_0,\bx) \phi(x) d\bx \lesssim_{t_0} \int \vert \nabla v_{k,1} \vert^2 (t, \bx) \phi(x) d\bx \\
			& + \int_{t_0-1}^{t_0} \int v_{k,1} (t',\bx) \phi(x) d\bx dt' + \sup_{t' \in [t_0-1,t_0]} \int v_{k,1} (t',\bx)^2 \phi(x) d\bx 
		\end{align*}
		and a second integration from $t_0-1$ to $t_0$ combined with \eqref{eq:step6_small_int} concludes the proof of the lemma.

		\vspace{0.5cm}
  \noindent\textit{Seventh step: Exponential decay of $\tilde{u}_1$ on the left.}
    \begin{lemm} \label{lemm:mono_left}
			We have that, for any $t\in \mathbb{R}$ and $x_0>0$,
			\begin{align*}
				\int \tilde{u}_1^2\left(t,\bx + \tilde{\bz}_1(t) \right) \left( 1- \phi(x+x_0) \right) d\bx \lesssim e^{-\frac14 \sqrt{\underline{c}} x_0}.
			\end{align*}
			Furthermore, for any $t\in \mathbb{R}$ and $x<0$,
			\begin{align*}
				\int_y \tilde{u}_1^2\left(t,\bx+ \tilde{\bz}_1(t) \right) dy \lesssim e^{-\frac14 \sqrt{\underline{c}} x}.
			\end{align*}
		\end{lemm}
		
		Notice that in \eqref{eq:monotonic_mass2}, the bound was obtained by using the monotonicity of the mass on the right of the first solitary wave. A similar monotonicity result holds for the mass of $u$ on the left of the first solitary wave with a weight going further from the first solitary wave forward in time but always staying on the right of the second wave (see Picture \eqref{fig:evol_left_first}). We refer to the proof of Lemma \ref{lemm:monotonicity_mass} with a weight $\psi_\gamma( x- z_1(t)+\gamma (t-t_0) +x_0)$. 
		
		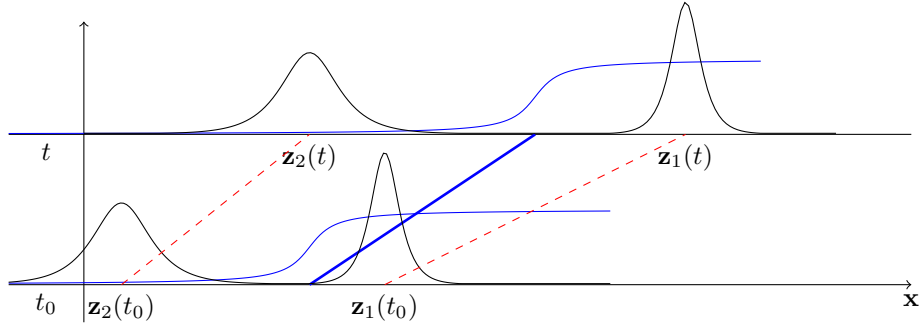
\begin{figure}[ht]
			\centering
			\begin{tikzpicture}[declare function = {
					solwave(\c,\z,\t) =(\c^(0.7))/ (exp(-max(-6,min(\c*(\t-\z),6))) +exp(max(-6,min(\c*(\t-\z),6))));
					weight(\z,\t) = 0.5+atan((\t-\z)*4)/180;}]
				
				\draw [->] (-1,0) -- (11,0);
				\draw (11,0) node[below]{$\bx$};
				\draw [->] (0,-0.5) -- (0,3.5);
				
				\draw (-0.5,0) node[below]{$t_0$};
				\draw (-1,2) -- (11,2);
				\draw (-0.5,2) node[below]{$t$};
				
				\draw [samples=\Num,blue,domain=-1:7] plot [variable=\t] ({\t},{weight(3,\t)});
				\draw [samples=\Num,blue,domain=-1:9] plot [variable=\t] ({\t},{2+weight(6,\t)});
				\draw[blue,line width=1] (3,0) -- (6,2);
				
				\draw [samples=\Num,domain=-1:7] plot [variable=\t] ({\t},{solwave(3,0.5,\t)+solwave(6,4,\t)});
				\draw (0.5,0) node[below]{$\bz_2(t_0)$};
				\draw (4,0) node[below]{$\bz_1(t_0)$};
				
				\draw [samples=\Num,domain=0:10] plot [variable=\t] ({\t},{2+solwave(3,3,\t)+solwave(6,8,\t)});
				\draw (3,2) node[below]{$\bz_2(t)$};
				\draw (8,2) node[below]{$\bz_1(t)$};
				
				\draw[red,dashed] (0.5,0) -- (3,2);
				\draw[red,dashed] (4,0) -- (8,2);
				
			\end{tikzpicture}
		    \caption{Evolution of the localization of the weighted mass on the left of the first solitary wave.}
		    \label{fig:evol_left_first}
		\end{figure}
		
		With this monotonicity result in hand, we conclude the proof of Lemma \eqref{lemm:mono_left} by using \eqref{eq:step4.1.1}and  arguing exactly as in the proof of Lemma 4.10 in \cite{CMPS16}.
		
		We conclude the proof of Proposition \ref{propo:limit_object_1} by gathering the seven intermediate steps.

	\subsection{Proof of Proposition \ref{propo:limit_object_2}}
		The proof of this proposition follows the same ideas of the proof of Proposition \ref{propo:limit_object_1}. We paint the proof in broad brush detailing the differences. 
		
		Let $\{t_n\}_n$ a sequence which tends to $+\infty$. Since $ \left(u -R_1 \right)(t,\cdot +\bz_2(t))$ is bounded in $H^1$ and $c_2(t)$ is bounded in $\mathbb{R}$, there exist a subsequence $\{t_n\}_{n_k}$, a limit object $\tilde{u}_{0,2}$ and a speed $\tilde{c}_{0,2}$ such that
		\begin{align*}
			\left( u -R_1 \right) \left(t_{n_k}, \cdot + \bz_2(t_{n_k}) \right)- \tilde{u}_{0,2} \underset{H^1}{\rightharpoonup} 0 \quad \text{and} \quad c_2(t_{n_k}) \rightarrow \tilde{c}_{0,2}.
		\end{align*}
	Furthermore, by using the orbital stability (see Theorem \ref{theo:orbital} (i)), it follows that
	\begin{align*}
		\left\| \tilde{u}_{0,2} - Q_{\tilde{c}_{0,2}} \right\|_{H^1} \leq \liminf_{k \to +\infty} \left\| \left( u -R_1 -R_2 \right) \left(t_{n_k}, \cdot + \bz_2(t_{n_k}) \right) \right\|_{H^1} \leq \beta. 
	\end{align*}
Let us denote  by $\tilde{u}_2$ the solution to \ref{ZK} with initial condition $\tilde{u}_{2}(0) = \tilde{u}_{0,2}$. By the local well-posedness $\tilde{u}_{2}\in \mathcal{C}\left( \mathbb{R} : H^1(\mathbb{R}^2) \right)$ and by the orbital stability result around one solitary wave in Theorem \ref{theo:orbital_stab}, it holds
\begin{align*}
	\sup_{t \in \mathbb{R}} \left\| \tilde{u}_2(t, \cdot + \tilde{\bz}_2(t)) -Q_{\tilde{c}_{0,2}}\right\|_{H^1} \lesssim \beta, 
\end{align*}
where $\tilde{\bz}_2=(\tilde{z}_2,\tilde{\omega}_2)$ the modulation parameter of translation. 

As before, we split the proof into different steps.

\vspace{0.5cm}
\noindent\textit{First step: Monotonicity property of $u$ on the right of the second solitary wave.}
We state a lemma on the decay of the $H^1$-norm on the right of the second solitary wave, using the function $\phi$ defined in \eqref{defi:phi}.

\begin{lemm}\label{lemm:decay_u_2}
	Let $x_0>0$ and $\vert \theta_1 \vert < \frac{\pi}{3}$. It holds
	\begin{align*}
		\limsup_{t \rightarrow +\infty} \int \left[ \left\vert \nabla \left( u - R_1 \right) \right\vert^2 + ( u - R_1)^2 \right] \left(t, \bx + \bz_2(t) \right)  \phi (x+\tan(\theta_1)y-x_0) dx \lesssim e^{-\frac14 \sqrt{\underline{c}} x_0}.
	\end{align*}
\end{lemm}

When working on the right of the second solitary wave, we have to take into account the mass of the first solitary wave (observe the difference with Lemma \ref{lemm:decay_u_1}).

To prove Lemma \ref{lemm:decay_u_2}, we state a monotonicity property on the right of the second solitary wave (see also Section 6 in \cite{Mol18} for an alternative definition of the lines between two peakons).

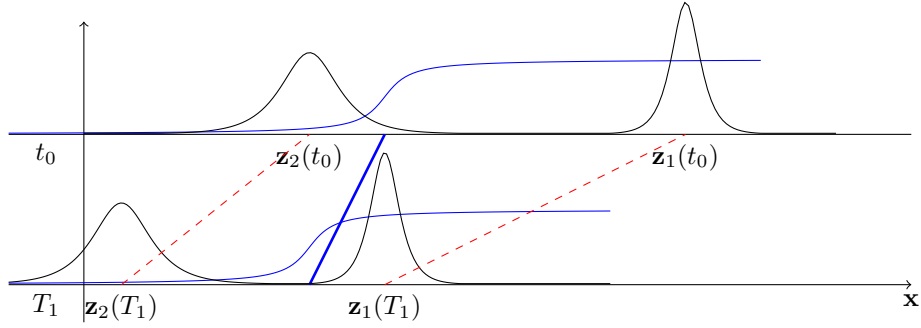
\begin{figure}[ht]
	\centering
	\begin{tikzpicture}[declare function = {
			solwave(\c,\z,\t) =(\c^(0.7))/ (exp(-max(-6,min(\c*(\t-\z),6))) +exp(max(-6,min(\c*(\t-\z),6))));
			weight(\z,\t) = 0.5+atan((\t-\z)*4)/180;}]
		
		\draw [->] (-1,0) -- (11,0);
		\draw (11,0) node[below]{$\bx$};
		\draw [->] (0,-0.5) -- (0,3.5);
		
		\draw (-0.5,0) node[below]{$T_1$};
		\draw (-1,2) -- (11,2);
		\draw (-0.5,2) node[below]{$t_0$};
		
		\draw [samples=\Num,blue,domain=-1:7] plot [variable=\t] ({\t},{weight(3,\t)});
		\draw [samples=\Num,blue,domain=-1:9] plot [variable=\t] ({\t},{2+weight(4,\t)});
		\draw[blue,line width=1] (3,0) -- (4,2);
		
		\draw [samples=\Num,domain=-1:7] plot [variable=\t] ({\t},{solwave(3,0.5,\t)+solwave(6,4,\t)});
		\draw (0.5,0) node[below]{$\bz_2(T_1)$};
		\draw (4,0) node[below]{$\bz_1(T_1)$};
		
		\draw [samples=\Num,domain=0:10] plot [variable=\t] ({\t},{2+solwave(3,3,\t)+solwave(6,8,\t)});
		\draw (3,2) node[below]{$\bz_2(t_0)$};
		\draw (8,2) node[below]{$\bz_1(t_0)$};
		
		\draw[red,dashed] (0.5,0) -- (3,2);
		\draw[red,dashed] (4,0) -- (8,2);
		
	\end{tikzpicture}
	\caption{Evolution of the weight function at times $t_0$ and $T_1$ }
\end{figure}

We define
    \begin{align*}
		\hat{x} =\hat{x}\left( x_0,\theta_1, t_0,t,\bx \right) := x-z_2(t_0)-x_0 +\frac{\underline{c}}{2}(t_0-t) +\tan(\theta_1)(y-\omega_2(t_0)),
  \end{align*}
and
  \begin{align}
  \hat{I}_{x_0, t_0,\theta_1} (t) := \int u^2(t,\bx) \phi (\hat{x}) d\bx, \quad \hat{J}_{x_0,t_0,\theta_1}(t) := \int \left( \left\vert \nabla u \right\vert^2 (t,\bx) - \frac23 u^3(t,\bx) \right)\phi(\hat{x}) d\bx.\label{defi:asympt_stab2}
	\end{align}

Inspired from the proof of Lemma \ref{lemm:monotonicity_2}, we define the following distances
\begin{align*}
    \hat{d}_1(t_0,t) & := z_1(t) - z_2(t_0)-x_0 + \frac{\underline{c}}{2}(t_0-t) + \tan(\theta_1) \left( \omega_1(t)- \omega_2(t_0) \right), \\
    \hat{d}_2(t_0,t) & := - z_2(t) + z_2(t_0) + x_0 - \frac{\underline{c}}{2} \left( t_0-t \right) - \tan (\theta_1) \left( \omega_2(t) - \omega_2(t_0) \right).
\end{align*}

\begin{lemm}\label{lemm:limsup2_mass}
    Let $\left\vert \theta_1 \right\vert < \frac{\pi}{3}$. There exist $C=C(\theta_1)$ and $\kappa>0$ such that, for any $x_0>0$ the following holds. There exists $T_0= T_0(\theta_1,x_0)$ such that for any $t_0>T_0$, by defining $T_1$ the unique time such that $\hat{d}_1(t_0,T_1)=x_0$, it holds $ T_1\geq \kappa t_0$ and for any $t\in [T_1, t_0]$,
	\begin{align}\label{eq:mono_mass_in_between}
		\hat{I}_{x_0, t_0,\theta_1} (t_0) - \hat{I}_{x_0, t_0,\theta_1} (t) \leq C e^{-\frac14 \sqrt{\underline{c}} x_0}, \quad \hat{J}_{x_0,t_0,\theta_1}(t_0) - \hat{J}_{x_0,t_0,\theta_1}(t) \leq C e^{-\frac14 \sqrt{\underline{c}} x_0} .
	\end{align}
\end{lemm}

\begin{proof}
	Let us fix $\left\vert \theta_1 \right\vert < \frac{\pi}{3}$ and $x_0>0$. First, for $t_0>T_0$ to be fixed later, we define a positive time $T_1$ such that on a time interval $[T_1,t_0]$, the weight function is located between the centers $\bz_i(t)$ of the two scaled solitary waves and at least at a distance $x_0$, see Figure \ref{fig:position_D_t0_T1}.

\begin{figure}[h]
	\centering
	\begin{minipage}{.57\textwidth}
	\begin{tikzpicture}		
		
		\draw (9,2.3) node[left]{at time $t_0$};
		\draw [->] (-0.5,2) -- (9,2);
		\draw (9,2) node[below]{$x$};
		\draw [->] (0,1.25) -- (0,3);
		\draw (0,3) node[left]{$y$};
		
		\draw (3,2.6) circle(0.5);
		\draw (2.5,2.8) node[left]{$R_2$};
		\draw (8,1.5) circle(0.2);
		\draw (8.3,1.5) node[right]{$R_1$};
		
		\draw[blue] (3.8,3) -- (4.8,1);
		\draw[dashed] (4.3,2) -- (4.3,3);
		\draw[->] (4.3,2.7) arc (90:115:0.7);
		\draw (4.4,2.9) node[left]{$\theta_1$};
		\draw[red] (3,2.6) -- (4,2.6);
		\draw[red] (3.8,2.6) node[below]{$x_0$};	
		
		\draw (9,0.3) node[left]{at time $T_1(t_0)$};
		\draw [->] (-0.5,0) -- (9,0);
		\draw (9,0) node[below]{$x$};
		\draw [->] (0,-0.75) -- (0,1);
		\draw (0,1) node[left]{$y$};
		
		\draw (0.4,0.6) circle(0.5);
		\draw (0.8,0.8) node[right]{$R_2$};
		\draw (4.4,-0.5) circle(0.2);
		\draw (4.7,-0.5) node[right]{$R_1$};
		
		\draw[blue] (3.5,-1) -- (2.5,1);
		\draw[dashed] (3,0) -- (3,1);
		\draw[->] (3,0.7) arc (90:115:0.7);
		\draw (3.1,0.9) node[left]{$\theta_1$};
		\draw[red] (4.4,-0.5) -- (3.2,-0.5);
		\draw[red] (3.6,-0.5) node[above]{$x_0$};				
	\end{tikzpicture}
	\caption{Localization of the weight function at times $t_0$ and $T_1$ with respect to the two solitary waves}
	\label{fig:position_D_t0_T1}
	\end{minipage}%
	\begin{minipage}{0.45\textwidth}
		\centering
		\begin{tikzpicture}		
			
			\draw (0,0) circle(0.2);
			\draw (0.2,0) node[right]{$R_1$};
			
			\draw (1,0) -- (-2,1.5);
			\draw (-2,1.75) node[left]{$\mathcal{D}_-$};
			\draw[dashed] (1,0) -- (1,1);
			\draw[->] (1,0.7) arc (90:155:0.7);
			\draw (0.8,0.7) node[above]{$\theta_0$};
			\draw (1,0) -- (-2,-1.5);
			\draw (-2,-1.5) node[left]{$\mathcal{D}_+$};
			\draw[dashed] (1,0) -- (1,-1);
			\draw[->] (1,-0.7) arc (-90:-155:0.7);
			\draw (0.8,-0.7) node[below]{$\theta_0$};	
			\draw[blue] (0.5,-1.5) -- (-1.7,2);
			\draw[blue] (0.5,-1.5) node[right]{$\mathcal{D}$};
			\draw[blue,dashed] (-1.3,1.4) -- (-1.3,2.1);
			\draw[blue,->] (-1.3,1.9) arc (90:115:0.7);
			\draw[blue] (-1.5,1.9) node[above]{$\theta_1$};
			\draw[red] (0,0) -- (-1.7,0);
			\draw[red] (-1,0) node[below]{$B$};		
			\draw[red] (-1.7,-1.35) -- (-1.7,1.35);	
			\draw (-1.5,-1.25) node[above]{$\mathcal{T}$};
			\draw (-1.8,1) node[left]{$\textbf{a}$};
			\draw (-1.7,1.2) node[left]{$\times$};
		\end{tikzpicture}
	\caption{Positions of $\mathcal{D}$, $\mathcal{D}_-$ and $\mathcal{D}_+$ with respect to the first solitary wave.}
	\label{fig:position_D_D-_D+}
	\end{minipage}
\end{figure}

First, we want that the line $\hat{x}=0$ remains located between the two solitary waves with distances $\hat{d}_1$ and $\hat{d}_2$ larger or equal to $x_0$. Since $\hat{d}_2(t_0,t_0)=x_0$, we define $\tilde{T_0} := \inf \big\{ t_0\in \mathbb{R}_+ :  \hat{d}_1(t_0,t_0) >x_0 \big\}$.
Observe that $g(t_0) := \hat{d}_1(t_0,t_0)$ satisfies $\left\vert g'(t_0) - \sigma \right\vert \lesssim \alpha$, thus $g$ is increasing and tends to $+\infty$ by \eqref{eq:boundz}, which justifies the definition of $\tilde{T}_0$.

 We now define the time $T_0:= \inf \{t_0>\tilde{T}_0 : \ \hat{d}_1(t_0,0)<0\}$. By \eqref{eq:ODE} and \eqref{eq:c_n}, computing $\frac{d}{dt_0} \hat{d}_1(t_0,0) \leq -c_2^0 + \frac14 \sigma + \frac12 \underline{c} < -\frac14 \underline{c} $, one obtains that $\hat{d}_1(t_0,0)<0$ for any $t_0$ large enough, which justifies the definition of $T_0$.
 
Next, for $t_0 > T_0$, we define $ T_1 =T_1(t_0)$ such that $\hat{d}_1(t_0,T_1) =x_0$. Then, it follows that, for any $t \in [T_1,t_0]$,
 \begin{align}\label{cond:distances}
     \hat{d}_1(t_0,t)  \ge x_0 \quad \text{and} \quad \hat{d}_2(t_0,t) \ge x_0.
 \end{align}
 Indeed, from \eqref{eq:bound_dot:z_omega_i_t} and $\beta \leq \kappa \underline{c}$ with an adequate $\kappa$, the two distance functions are monotonic
 \begin{align}\label{eq:derivee_hat_d_1_d_2}
     \left\vert \partial_t \hat{d}_1(t_0,t) - \left( c_1^0  - \frac{\underline{c}}{2} \right) \right\vert \leq \frac{\underline{c}}{16} \quad \text{and} \quad \left\vert  \partial_t \hat{d}_2(t_0,t) + c_2^0 - \frac{\underline{c}}{2} \right\vert \leq \frac{\underline{c}}{16},
 \end{align}
 which implies \eqref{cond:distances}.

 To prove the lower bound on $T_1$, an integration of \eqref{eq:derivee_hat_d_1_d_2} from $T_1$ to $t_0$ gives
 \begin{align*}
     \left\vert \hat{d}_1(t_0,t_0) - x_0 - \left( c_1^0 - \frac{\underline{c}}{2}\right) (t_0-T_1) \right\vert \leq \frac{\underline{c}}{16} (t_0-T_1)
 \end{align*}
 and since $\hat{d}_1(t_0,t_0) \leq Z + (1+\frac{1}{16}) \sigma t_0$ by \eqref{eq:encad_z}, we obtain, recalling that $\sigma<\underline{c}$, that $T_1 \geq \kappa t_0$ for some $\kappa>0$.

We prove the monotonocity formula \eqref{eq:mono_mass_in_between} for $\hat{I}_{x_0,t_0,\theta_1}$. The computations are similar to the ones in the proof of \eqref{eq:monotonic_mass2}, the only difference relies in estimating the cubic term. Observe that an integration of $\partial_t \hat{d}_1(t_0,t) > \frac14 \underline{c}$ (see \eqref{eq:derivee_hat_d_1_d_2}) from $T_1$ to $t$ provides $\hat{d}_1(t_0,t) > x_0 + \frac14 \underline{c} (t-T_1)$ and an integration of $\partial_t \hat{d}_2(t_0,t) < \frac14 \underline{c}$ from $t$ to $t_0$ provides $\hat{d}_2(t_0,t)> x_0 - \frac14 \underline{c} (t_0-t)$. Then, inspired from \eqref{sol_wave_weight}, we have
\begin{align*}
		\left\| R_1(t) \phi'(\hat{x}) \right\|_{L^\infty} \lesssim e^{-\frac14 \sqrt{\underline{c}} \left(x_0 +\frac{\underline{c}}{4} (t-T_1)\right)} \ \text{and} \ \left\| R_2(t) \phi'(\hat{x}) \right\|_{L^\infty} \lesssim e^{-\frac14 \sqrt{\underline{c}} \left(x_0 -\frac{\underline{c}}{4} (t_0-t)\right)}.
\end{align*}
This implies, arguing as in \eqref{eq:cubic_term_2}, the following bound on the non-linear term
\begin{align*}
		\left\vert \frac43 \int u^3 \phi'(\hat{x}) \right\vert \lesssim \| \epsilon \|_{H^1} \left( \int \left( \vert \nabla u \vert^2 + u^2 \right) \phi'(\hat{x}) \right) + e^{-\frac14 \sqrt{\underline{c}} \left( x_0 + \frac14 \underline{c} (t_0-t)\right)} + e^{-\frac14\sqrt{\underline{c}} \left( x_0 + \frac14 \underline{c} (t-T_1)\right)} .
	\end{align*}
The rest of the proof of \eqref{eq:mono_mass_in_between} is similar to the one of \eqref{eq:monotonic_mass}.

The proof of the formula for $\hat{J}_{x_0,t_0,\theta_1}$ is obtained by combining the ideas explained above with the computations in Lemma 3.4 in \cite{CMPS16}.
\end{proof}

\begin{proof}[Proof of Lemma \ref{lemm:decay_u_2}]
We only deal with the $L^2$-weighted norm, the case of the gradient is similar. We first claim that, for $t_0>T_0$ large enough,
\begin{align}
    \MoveEqLeft
    \int \left( u - R_1 \right)^2 \left( t_0, \bx + \bz_2(t_0) \right) \phi( x- \tan \theta_1 y -x_0) d \bx \nonumber \\
    & \lesssim \int u^2(T_1, \bx + \bz_1(T_1) ) \phi( x- \tan \theta_1 y + x_0) d\bx - \int R_1^2(T_1,\bx) d \bx + e^{- \frac14 \sqrt{\underline{c}}x_0}. \label{eq:limsup_t0_T1}
\end{align}

Indeed, it holds
\begin{align}
\MoveEqLeft
	\int \left[ \left( u - R_1 \right)^2 -u^2 + R_1^2 \right](t_0, \bx + \bz_2(t_0)) \phi \left(x- \tan(\theta_1)y -x_0 \right) d\bx \nonumber \\
	& =2 \int \left[ (R_1 - u )R_1 \right](t_0, \bx + \bz_2(t_0)) \phi (x- \tan(\theta_1)y -x_0) d\bx . \label{eq:dvt_square} 
\end{align} 

Since $Q_{c_1(t_0)}$ decays exponentially and $u(t_0, \cdot + \bz_1(t_0)) - Q_{c_1(t_0)}$ converges to $0$ in $L^2(x>-B)$ by \eqref{propo:asymp:1soliton.3}, the right-hand side of \eqref{eq:dvt_square} can be chosen for $t_0$ large enough lower than $e^{- \frac14 \sqrt{\underline{c}} x_0}$.
Next, we notice as in \eqref{ineq:R_psi_L2} that
\begin{align}\label{eq:Q_c1_approx}
    \left\vert \int R_1^2(t_0,\bx + \bz_2(t_0)) \phi \left( x - \tan(\theta_1)y -x_0 \right) - \int R_1^2(t_0) \right\vert \lesssim e^{- \frac14 \sqrt{\underline{c}}x_0}.
\end{align}
On the other, by recalling $\hat{d}_1(t_0,T_1)=x_0$, \eqref{eq:mono_mass_in_between} yields
\begin{align}
    \MoveEqLeft
    \int u^2(t_0, \bx + \bz_2(t_0) ) \phi \left( x- \tan (\theta_1) y -x_0 \right) d\bx - \int u^2(T_1, \bx + \bz_1(T_1) ) \phi(x- \tan(\theta_1) y + x_0) d\bx \nonumber \\
    & \lesssim e^{- \frac14 \sqrt{\underline{c}} x_0}. \label{eq:monotonic_t0_T1}
\end{align}
Gathering \eqref{eq:dvt_square}, \eqref{eq:Q_c1_approx} and \eqref{eq:monotonic_t0_T1} and using that $c_1(t)$ converges at infinity conclude the proof of \eqref{eq:limsup_t0_T1}. 

To conclude the proof, it suffices to prove that the $\limsup$ as $t_0$ (and thus $T_1$) goes to $\infty$ of the right-hand side of \eqref{eq:limsup_t0_T1} is bounded by $Ce^{-\frac18 \sqrt{\underline{c}}x_0}$. Let $\vert \theta_0 \vert \in (\vert \theta_1 \vert ; \frac{\pi}{3})$ and define the three lines $\mathcal{D}$, $\mathcal{D}_+$, $\mathcal{D}_-$ at time $T_1(t_0)$ around the first solitary wave (see Figure \ref{fig:position_D_D-_D+}) by
\begin{align*}
	& \mathcal{D}_\pm:= \left\{ x- z_1(T_1) \pm \tan(\theta_0) \left( y - \omega_1(T_1) \right) - x_0 =0\right\}, \\
	& \mathcal{D} := \left\{ x- z_1(T_1) -\tan(\theta_1) \left( y - \omega_1(T_1) \right)- x_0 =0 \right\}.
\end{align*}
Now, as explained in Figure \ref{fig:position_D_D-_D+}, choose $B=B(x_0,\theta_0,\theta_1)>0$ such that any point ${\bf a}=(a,b)$ below $\mathcal{D}_-$ and above $\mathcal{D}_+$ with $a<-B$ is closer to $\mathcal{D}_-$ or $\mathcal{D}_+$ than to $\mathcal{D}$ and denote the triangle $$\mathcal{T}=\left\{\bx : x>-B; x< z_1(t) \pm \tan(\theta_1) \left( y - \omega_1(t) \right) + x_0 \right\}. $$ 
Since $\phi$ is strictly increasing, it follows from this construction that 
\begin{align*}
	\MoveEqLeft
	\int u^2\left(T_1, \bx + \bz_1(T_1) \right) \phi (x -\tan(\theta_1)y+ x_0) d\bx \leq \int_{\mathcal{T}} u^2 \left(T_1, \bx + \bz_1(T_1) \right) d\bx \\
	& \quad + \int u^2\left(T_1, \bx + \bz_1(T_1) \right) \left( \phi \left( x + \tan(\theta_0) y- x_0 \right) +\phi \left( x - \tan(\theta_0) y- x_0 \right) \right) d\bx.
\end{align*}
Therefore, we conclude the proof of Lemma \ref{lemm:decay_u_2} gathering \eqref{eq:limsup_t0_T1}, \eqref{eq:limsup_T1} and the strong convergence in \eqref{propo:asymp:1soliton.3}, and recalling the lower bound on $T_1\geq \kappa t_0$.
\end{proof}

\vspace{0.5cm}
\noindent \textit{Second step: Strong $L^2$ convergence of $\left( u - R_1 \right) (t_{n_k}, \cdot + \bz_2(t_{n_k}))$ to $\tilde{u}_{0,2}$ on the right.}

\begin{lemm}
	It holds, for any $B>0$,
	\begin{align*}
		v_{k,2,0}:= \left( u - R_1 \right) (t_{n_k}, \cdot + \bz_2(t_{n_k}) ) - \tilde{u}_{0,2 } \underset{k\rightarrow +\infty}{\rightarrow} 0 \quad \text{in} \quad L^2(x>-B).
	\end{align*}
\end{lemm}
The proof is the same as in Lemma \ref{lemm:step2} by using Lemma \ref{lemm:decay_u_2} instead of Lemma \ref{lemm:decay_u_1}.

\vspace{0.5cm}
\noindent \textit{Third step: Exponential decay of $\tilde{u}_{0,2}$ on the right on finite time intervals.}

\begin{lemm} \label{lemm:step3_2}
	For any $x_0>0$, it holds
	\begin{align*}
		\int \left( \vert \nabla \tilde{u}_{0,2} \vert ^2 + \tilde{u}_{0,2}^2 \right) (\bx) \phi (x-x_0) d\bx \lesssim e^{-\frac14 \sqrt{\underline{c}}x_0}.
	\end{align*}
	Moreover, for all $t_0\geq 0$, there exists $K(t_0)>0$ such that
	\begin{align*}
		\sup_{t\in [0,t_0]} \int \left(\vert \nabla \tilde{u}_2\vert^2 + \tilde{u}_2^2 \right) \left(t, \bx \right) e^{\frac14 \sqrt{\underline{c}}x} dx \leq K(t_0).
	\end{align*}
\end{lemm}
By weak convergence \eqref{eq:weak_conv_1}, it holds for any $x_0>0$
\begin{align*}
\left\| \tilde{u}_{0,2} \sqrt{\phi (\cdot-x_0)} \right\|_{H^1} \leq \liminf_{k} \left\| \left( u - R_1 \right) \left(  t_{n_k}, \cdot +  \bz_2(t_{n_k}) \right)  \sqrt{\phi (\cdot-x_0)} \right\|_{H^1}.
\end{align*} 
We conclude the first part of the lemma gathering the previous inequality with Lemma \ref{lemm:decay_u_2}. 

The second estimate in Lemma \ref{lemm:step3_2} is obtained by a Gronwall argument arguing exactly as in Lemma \ref{lemm:step3_1}.

\vspace{0.5cm}
\noindent \textit{Fourth step: Strong $L^2$ convergence of $\left( u - R_1 \right) (t_{n_k}+t, \cdot+ \bz_2(t_{n_k}))$ to $\tilde{u}_2(t)$ on the right.}

\begin{lemm}\label{lemm:step4_2}
	For any $B>0$ and $t \in \mathbb{R}$, it holds
	\begin{align*}
		\left( u -R_1 \right) \left( t_{n_k}+t, \cdot + \bz_2(t_{n_k}) \right) - \tilde{u}_2(t) \underset{k \rightarrow + \infty}{\rightarrow} 0 \quad \text{in} \quad L^2(x>-B)
	\end{align*}
	and for any $t\in \mathbb{R}$
	\begin{align*}
		\left( u -R_1 \right) \left( t_{n_k}+t, \cdot + \bz_2(t_{n_k}) \right) - \tilde{u}_2(t) \underset{k \rightarrow + \infty}{\rightharpoonup} 0 \quad \text{in} \quad H^1(\mathbb{R}^2).
	\end{align*}
	Moreover, with $\tilde{\bz}_2=(\tilde{z}_2, \tilde{\omega}_2)$ the $\mathcal{C}^1$-function associated to the modulation parameters of $\tilde{u}_2$, it holds
	\begin{align*}
		\bz_2(t+t_{n_k}) - \bz_2(t_{n_k}) \underset{k \to +\infty}{\rightarrow}\tilde{\bz}_2(t) \quad \text{and} \quad \tilde{\bz}_2(0)=0.
	\end{align*}
\end{lemm}

Let us define
  \begin{align}\label{defi:v_k2}
      v_{k,2}(t,\bx) := \left( u - R_1 \right) \left(t+t_{n_k},\bx+ \bz_2(t_{n_k}) \right) -\tilde{u}_2(t,\bx).
  \end{align}
  We argue by using a Gronwall argument as in the proof of Theorem 3 (b) in \cite{MM08} to obtain that
  \begin{align*}
			\sup_{t \in [0,t_0]}\int_{\mathbb{R}^2} v_{k,2}(t,\bx)^2 \phi (x) d\bx \leq \tilde{K}(t_0) \int_{\mathbb{R}^2} v_{k,2}(0,\bx)^2 \phi (x) d\bx.
		\end{align*} 
  The rest of the proof is then similar to the proof of Lemma \ref{lemm:step4.1.1}.

\vspace{0.5cm}
Since the next steps are similar to the ones around the first rescaled limit object, we state the lemmas around the second limit object without giving the proofs. 

\vspace{0.5cm}
\noindent \textit{Fifth step: Exponential decay of $\tilde{u}_{2}$ on the right.}

\begin{lemm}
	For all $t \in \mathbb{R}$ and $x_0>0$, it holds
	\begin{align*}
		\int \left( \vert \nabla \tilde{u}_2 \vert^2 + \tilde{u}_2^2 \right) \left(t, \bx +\tilde{\bz}_2(t) \right) \phi (x-x_0) d\bx \lesssim e^{-\frac14 \sqrt{\underline{c}}x_0}.
	\end{align*}
	Furthermore, for any $t>0$ and $x>0$,
	\begin{align*}
		\int_y \tilde{u}_2 \left(t, \bx+ \tilde{\bz}_2(t) \right) ^2 dy \lesssim e^{-\frac14 \sqrt{\underline{c}}x}.
	\end{align*}
\end{lemm}

\vspace{0.5cm}
\noindent \textit{Sixth step: Strong $H^1$-convergence of $\left( u - R_1 \right) \left(t_{n_k}+t, \cdot + \bz_2(t_{n_k}) \right)$ to $\tilde{u}_{2}(t)$ on the right.}

\begin{lemm}
	For any $t\in \mathbb{R}$ and $B>0$, it holds
	\begin{align*}
		\left(u -R_1 \right) \left( t_{n_k}+t, \cdot + \bz_2(t_{n_k}) \right) - \tilde{u}_{2}(t) \underset{n \rightarrow + \infty}{\rightarrow} 0 \quad \text{in} \quad H^1(x>-B).
	\end{align*}
\end{lemm}

\vspace{0.5cm}
\noindent \textit{Seventh step: Exponential decay of $\tilde{u}_2$ on the left.}

\begin{lemm}
	We have that, for any $t\in \mathbb{R}$ and $x_0>0$,
	\begin{align*}
		\int \tilde{u}_2^2\left(t,\bx + \tilde{\bz}_2(t) \right) \left( 1- \phi (x+x_0) \right) d\bx \lesssim e^{-\frac14 \sqrt{\underline{c}}x_0}.
	\end{align*}
	Furthermore, for any $t\in \mathbb{R}$ and $x<0$,
	\begin{align*}
		\int_y \tilde{u}_2^2\left(t,\bx+ \tilde{\bz}_2(t) \right) dy \lesssim e^{- \frac14 \sqrt{\underline{c}}x}.
	\end{align*}
\end{lemm}

\section{The three dimensional case} \label{Sec:3d_case}

In this section, we explain how to prove Theorem \ref{theo:orbital_stab} in the case of the dimension $d=3$. The proof of the orbital stability is exactly the same as in the $2$-dimensional. Hence we will focus on the proof of the asymptotic stability. We introduce the notation $\bx=(x,\by)=(x,y_1,y_2) \in \mathbb R^3$. 

The proof is similar to the one in dimension $d=2$ in the former section.  The main difference is that the Liouville theorem proved in \cite{FHRY23} requires a slightly stronger decay assumption than the one in the $2$-dimensional case stated in  Theorem \ref{theo:nonlinear_liouville}. Let us recall the rigidity result of Proposition 2.7 in \cite{FHRY23}.

\begin{theo}
Let $0<\underline{c}< \tilde{c} < \bar{c}$. There exists $\alpha_{12}=\alpha_{12}(\underline{c}, \bar{c})$ and $A_{12}=A_{12}(\underline{c}, \bar{c})$ such that the following is true. Suppose that for $0< \alpha <\alpha_{12}$, $\tilde{v}\in \mathcal{C}(\mathbb{R}: H^1(\mathbb{R}^3))$ is a solution to \ref{ZK}$_{d=3}$ satisfying for some function $\bz_{\tilde{v}}(t) \in \mathbb{R}^3$,
\begin{align*}
    \forall t \in \mathbb{R}, \quad \left\| \tilde{v}\left( \cdot + \bz_{\tilde{v}}(t) \right) - Q_{\tilde{c}} \right\|_{H^1(\mathbb{R}^3)} < \alpha.
\end{align*}
By denoting $\tilde{c},\tilde{\bz}$ the modulation parameters, if $\tilde{v}$ satisfies the spatial decay property: for each $k \ge 0$, 
\begin{align} \label{decay:3d}
    \sup_{t \in \mathbb{R}} \int_{\bx \in \mathbb{R}^3} \langle \bx \rangle^{2k} \tilde{v}^2(t, \bx + \tilde{\bz}(t)) d\bx <\infty,
\end{align}
then there exists $\left\vert \check{c} -\tilde{c} \right\vert <A_{12} \alpha $ and $\check{\bz} \in \mathbb{R}^3$ such that
\begin{align*}
    \tilde{v}(t,\bx) = Q_{\check{c}} \left( \bx - \check{\bz} - \check{c} t \beone \right).
\end{align*}
\end{theo}

The condition \eqref{decay:3d} is a slightly stronger than the condition \eqref{decay:2d} of Theorem \ref{theo:nonlinear_liouville}. Thus, when taking the limit profiles around the two solitary waves, we have to ensure that they verify condition \eqref{decay:3d}. We explain how to ascertain this condition for the profile $\tilde{u}_1$ around the first solitary wave. We state an equivalent of Proposition \ref{propo:limit_object_1} in the $3$-dimensional case.

       \begin{propo}\label{propo:limit_object_1_3d}
		Under the assumptions of Theorem \ref{theo:orbital}, by taking $k=k(\underline{c},\bar{c})>0$ smaller and $K=K(\underline{c},\bar{c})>0$ larger if necessary, for any increasing sequence $\{t_n\}_n\rightarrow +\infty$, there exists a subsequence $\{t_{n_k}\}_k$ and $\tilde{u}_{0,1}\in H^1(\mathbb{R}^3)$ such that, for any $B>0$,
		\begin{align*}
			u \left( t_{n_k}, \cdot + \bz_1(t_{n_k}) \right) \underset{k\rightarrow +\infty}{\rightarrow} \tilde{u}_{0,1} \quad \text{in} \quad H^1(x >-B).
		\end{align*}
		
		Moreover, the solution $\tilde{u}_1$ of \ref{ZK} with initial condition $\tilde{u}_1(0)= \tilde{u}_{0,1}$ satisfies for any $t \in \mathbb{R}$
		\begin{align}\label{eq:u_tilde1_bound_3d}
			\left\| \tilde{u}_1 \left(t,\cdot + \tilde{\bz}_1(t) \right) - Q_{c_1^0} \right\|_{H^1} \lesssim \beta
		\end{align}
		with $\beta$ defined in \eqref{defi:alpha_star} and for any $t \in \mathbb R$, $|\theta_0| < \frac{\pi}3$ and $j=1,2$,
		\begin{align}\label{eq:u_tilde1_decay_3d}
			\int \tilde{u}_1^2  \left( t, \bx + \tilde{\bz}_1(t) \right) e^{\frac14\sqrt{\underline{c}}|x+\tan{\theta_0}y_j|} d\bx \lesssim 1. 
		\end{align}
		where $\tilde{\bz}_1$ corresponds to the function given by the modulation theory around one solitary wave. 
	\end{propo}

 \begin{proof}The proof follows the same steps than the one of Proposition \ref{propo:limit_object_1}. The monotonicity result in Lemma \ref{lemm:monotonicity_2} extends to the $3$-dimensional case by replacing $y$ by $y_j$ for $j=1,2$. Then one deduce that, for any $x_0>0$, 
 \begin{align*}
		\limsup_{t\rightarrow + \infty} \int \left( \vert \nabla u \vert^2 + u^2 \right) \left(t, \bx + \bz_1(t) \right) \phi \left( x+\tan{\theta_0}y_j- x_0 \right) d\bx \lesssim e^{- \frac14 \sqrt{\underline{c}} x_0}.
\end{align*}
This decay propagates to the limit profile $\tilde{u}_{0,1}$ and to its emanating solution $\tilde{u}_1$  as in the third, fourth and fifth steps. We deduce that, for any $t \in \mathbb R$, $|\theta_0|<\frac{\pi}3$, $j=1,2$, and $x_0>0$,
\begin{align*}
		 \int \left( \vert \nabla \tilde{u}_1 \vert^2 + \tilde{u}_1^2 \right) \left(t, \bx + \tilde{\bz}_1(t) \right) \phi \left( x+\tan{\theta_0}y_j- x_0 \right) d\bx \lesssim e^{- \frac14 \sqrt{\underline{c}} x_0} , 
\end{align*}
so that 
\begin{align*}
    \int \tilde{u}_1^2  \left( t, \bx + \tilde{\bz}_1(t) \right) e^{\frac14\sqrt{\underline{c}}(x+\tan{\theta_0}y_j)} d\bx \lesssim 1, 
\end{align*} 
by using that $e^{\frac14\sqrt{\underline{c}}(x+\tan{\theta_0}y_j-x_0)} \le \phi(x+\tan{\theta_0}y_j-x_0)$ for $x+\tan{\theta_0}y_j \le x_0$ and then letting $x_0 \to +\infty$. Finally, we deduce the decay on the left in \eqref{eq:u_tilde1_decay_3d} arguing as in the sixth and seventh steps. 
 \end{proof}

 With Proposition \ref{propo:limit_object_1_3d} in hand, we prove the $3$-dimensional version of Proposition \ref{propo:asymp:1soliton}. The only difference with the $2$-dimensional case consists in proving that $\tilde{u}_1$ satisfies \eqref{decay:3d}. We proceed as follows: observe that for $\bx \in \mathbb R^3$, $|\bx| \lesssim |x|+|x+y_1|+|x+y_2|$. Then an application of \eqref{eq:u_tilde1_decay_3d} with $\theta_0=0$ and $\theta_0=\frac{\pi}4$ for $j=1,2$ yields \eqref{decay:3d}. 

 The $3$-dimensional version of Proposition \ref{propo:asymp:2soliton} is proved similarly. Then the conclusion of the proof of Theorem \ref{theo:orbital_stab} (ii) in the $3$-dimensional case follows as in the $2$-dimensional case.

	\appendix

	\section{Proof of Lemmas \ref{est:R1R2} and \ref{lemma:est:Rpsi}} \label{app:R1R2}
	Before proving Lemma \ref{est:R1R2}, we state a technical lemma.
	\begin{lemm} \label{lemma:fg}
		Let $f,g \in C^{\infty}(\mathbb R^d)$ be two functions satisfying, for some $n \in \mathbb N$,
		\begin{equation} \label{def:fg}
			|f(\bx)|+|g(\bx)| \lesssim \langle \bx \rangle^n e^{-|\bx|}.
		\end{equation}
		Then, there exists a positive constant $c=c(n,d)>1$, such that for all $(z,\omega) \in \mathbb R^d$ with $z>c(n,d) $,  
		\begin{align} 
			\left|f(\bx-(z,\omega))g(\bx) \right| &\lesssim  e^{-z}(z^n+\langle (x,y-\omega) \rangle^n)\langle (x,y) \rangle^n, \label{est:fg:pointwise.1} \\ 
			\left|f(\bx-(z,\omega))g(\bx) \right| &\lesssim e^{-\frac{31}{32}z}\left( {\bf 1}_{\{x<0\}}e^{-\frac1{64}|(x,y- \omega)|}+{\bf 1}_{\{0<x<z\}}(e^{-\frac1{64}|y|}+e^{-\frac1{64}|y-\omega|}) \right. \nonumber\\ & \left.\quad \quad \quad \quad +{\bf 1}_{\{x>z\}}e^{-\frac1{64}|(x,y)|} \right) \label{est:fg:pointwise.2}
		\end{align}
		and, for all $1 \le p<\infty$,
		\begin{equation} \label{est:fg:Lp}
			\|f(\cdot-(z,\omega))g\|_{L^p} \lesssim e^{-\frac{15}{16}z}.
		\end{equation}
		
	\end{lemm}
	
	\begin{proof}
		Observe that \eqref{est:fg:Lp} is a direct consequence of \eqref{est:fg:pointwise.2}. We turn to the proof of \eqref{est:fg:pointwise.1} and \eqref{est:fg:pointwise.2}. In the region $x<0$, $|(x-z,y-\omega)| \ge |x-z|>z$ and $|(x-z,y-\omega))| \ge \frac1{\sqrt{2}}(|x|+|y-\omega|) \ge \frac1{\sqrt{2}}|(x,y-\omega)|$. Thus it follows that
		\begin{align*}
		  \left|f(\bx-(z,\omega))g(\bx) \right| & \lesssim e^{-|(x-z,y-\omega)|} \langle (x-z,y-\omega) \rangle^n \\ 
   & \lesssim \min\{ e^{-z}(z^n+\langle (x,y-\omega)\rangle^n),e^{-\frac{31}{32}z}e^{-\frac1{64}|(x,y-\omega))|}\} .
		\end{align*}
		In the region $x>z$, we argue similarly and use $|x|>z$, so that 
		\begin{equation*}
			\left|f(\bx-(z,\omega))g(\bx) \right| \lesssim e^{-|(x,y)|} \langle (x,y) \rangle^n \lesssim \min\{ e^{-z}(z^n+\langle (x,y) \rangle^n),e^{-\frac{31}{32}z}e^{-\frac1{64}|(x,y)|}\} .
		\end{equation*}
		Finally, in the region $0<x<z$, we have $|(x-z,y-\omega)| \ge z-x$ and $|(x,y)| \ge x$, so that 
		\begin{equation*}
			\left|f(\bx-(z,\omega))g(\bx) \right| 
			\lesssim \min\{ e^{-z}(z^n+\langle (x,y-\omega) \rangle^n)\langle (x,y) \rangle^n,e^{-\frac{31}{32}z}(e^{-\frac1{64}|y|}+e^{-\frac1{64}|y+\omega|})\} .
		\end{equation*}
		Therefore, gathering these estimates conclude the proof of \eqref{est:fg:pointwise.1} and \eqref{est:fg:pointwise.2}.
	\end{proof}
	
	\begin{proof}[Proof of Lemma \ref{est:R1R2}]
		The first set of identities follow directly from a rescaling argument. 
		
		Let us prove the first inequality in the second set of estimates. 
		By using the bound \eqref{asym:Q} and the notation $(z,\omega)=(z_1-z_2,\omega_1-\omega_2)$, we find  
		\begin{align*} 
			\left| \int R_1 R_2 \right| & \lesssim c_1c_2 \int e^{-\sqrt{c_1}|\bx-(z_1,\omega_1)|}e^{-\sqrt{c_2}|\bx-(z_2,\omega_2)|}d\bx  \\ 
			&=c_1 \int e^{-\sqrt{c_1c_2^{-1}}|\bx-\sqrt{c_2}(z,\omega)|}e^{-|\bx|} d\bx \\
			& \le c_1 \int e^{-|\bx-\sqrt{c_2}(z,\omega)|}e^{-|\bx|} d\bx , 
		\end{align*}
		so that \eqref{est:fg:Lp} and 
  \eqref{estimate_modulation:c:bis} conclude the proof of the first estimate on the right-hand side of \eqref{est:R1R2.1}.
		
		The proof of \eqref{est:R1R2.1}-\eqref{est:R1R2.5} follows arguing similarly.
		
		Finally we estimate \eqref{est:R1R2.6}. By using the bound \eqref{asym:Q} and \eqref{est:fg:pointwise.1}, we deduce that
		\begin{align*}
			\MoveEqLeft
			\left\vert \int \partial_x(R_1 R_2) R_1 \right\vert \\ & \lesssim c_1^2 c_2 \int e^{-2\sqrt{c_1}|\bx-(z_1,\omega_1)|}e^{-\sqrt{c_2}|\bx-(z_2,\omega_2)|}d\bx   \le c_1^3 \int e^{-2|\bx-\sqrt{c_2}(z,\omega)|}e^{-|\bx|} d\bx \lesssim c_1^3 e^{-\sqrt{c_2} z} .
		\end{align*}
		Estimate \eqref{est:R1R2.7} follows by arguing in a similar way.
	\end{proof}

\begin{proof}[Proof of Lemma \ref{lemma:est:Rpsi}]
    We first start by proving \eqref{est:Rpsi}. Observe from \eqref{asym:Q} and \eqref{prop:psi} that
			\begin{align*} 
				\left| R_1(\bx) \psi_\gamma'\left( x - m(t)\right) \right| 
            \lesssim \bar{c} e^{-\frac12\sqrt{\underline{c}}\left(\left| x-z_1(t)\right|+\left| x-m(t)\right| \right)}.
			\end{align*}
   In the case where $x<m(t)$, we have $\vert x-z_1(t) \vert=z_1(t)-x>\frac12 z$, while in the case $x>z_1(t)$, we have $\vert x-m(t) \vert=x-m(t)>\frac12 z$. Finally, in the case $x \in [m(t),z_1(t)]$ we have $ \vert x-z_1(t) \vert + \vert x- m(t) \vert = z_1(t)-x +x- m(t) = \frac12 z(t)$. Thus, we conclude from the above estimates and \eqref{eq:boundz} that 
            \begin{align*}
				\| R_1 \psi_\gamma'\left( \cdot - m(t)\right) \|_{L^\infty}  \lesssim \bar{c} e^{-\frac14 \sqrt{\underline{c}} z} \lesssim e^{-\frac14 \sqrt{\underline{c}} \left(Z+\sigma t\right)}
			\end{align*}
			The estimate for the second term on the left-hand side of \eqref{est:Rpsi} is obtained arguing similarly. 

   Next, we turn to the proof of \eqref{ineq:R_psi_L2}. Observe that, for $z>0$, $\omega\in \mathbb{R}$ and $c>0$, it holds
			\begin{align}\label{ineq:Qc_psi}
				\left\| Q_c \left( \cdot + z, \cdot - \omega \right) \psi \right\|_{L^2} \lesssim c^{\frac12} e^{ - \frac12\min\{1,\sqrt{c}\} z}.
			\end{align}
			Indeed, on the one hand, in the region where $x\leq -\frac{z}{2}$, we have $\psi(x) \lesssim e^{- \frac{z}{2}}$ (see \eqref{defi:psi}), and thus, it follows from \eqref{asym:Q} that
			\begin{align*}
				\left\| Q_c \left( \cdot +z, \cdot - \omega \right) \psi(\cdot) \right\|_{L^2(x<-\frac{z}2)}^2  \lesssim c^2 \iint_{x \leq -\frac{z}{2}} e^{- 2 \sqrt{c} \left\vert (x+z,y-\omega) \right\vert } e^{-z} dx dy \lesssim c e^{-z}. 
			\end{align*}
			On the other hand, in the region $x>-\frac{z}{2}$, we use $0\le \psi \le 1$, a polar coordinate change and the decay \eqref{asym:Q} to deduce
			\begin{align*}
				\left\| Q_c \left( \cdot + z, \cdot - \omega \right) \psi \right\|_{L^2(x > - \frac{z}{2})}^2 
				& \lesssim c^2 \int_{r>\frac{z}{2}} \frac{1}{r} e^{- 2 \sqrt{c}r} r dr \lesssim c e^{-\sqrt{c}z}.
			\end{align*}

			We now prove \eqref{ineq:R_psi_L2}. By using $\psi(-x)=1-\psi(x)$, \eqref{ineq:Qc_psi} and $4c_1>\gamma$ (see \eqref{eq:c_n}), we have 
            \begin{align*}
				\left\| R_1 \left( \psi_\gamma (\cdot-m)- 1 \right) \right\|_{L^2} = \frac{\sqrt{\gamma}}2 \left\| Q_{\frac{4c_1}{\gamma}} \left( \cdot + \frac{\sqrt{\gamma}}4 z, \cdot - \frac{\sqrt{\gamma}}2\omega_1  \right) \psi \right\|_{L^2} \lesssim \bar{c}^{\frac12} e^{ -\frac18 \sqrt{\gamma}z} .
			\end{align*}
           Similarly, the same estimate holds for $\|R_2\psi_s(\cdot-m)\|_{L^2}$ in the case $4c_2 \geq \gamma$. In the case where $4c_2<\gamma$, we obtain from \eqref{ineq:Qc_psi} that 
			\begin{align*}
				\left\| R_2 \psi_\gamma \right\|_{L^2} = \frac{\sqrt{\gamma}}2 \left\| Q_{\frac{4c_2}{\gamma}} \left( \cdot +\frac{\sqrt{\gamma}}4 z , \cdot -\frac{\sqrt{\gamma}}4 \omega_2 \right) \psi \right\|_{L^2} \lesssim \underline{c}^\frac12  e^{-\frac14 \sqrt{c_2}z}.
			\end{align*}
			Therefore, we conclude the proof of \eqref{ineq:R_psi_L2} by combining these estimates with \eqref{estimate_modulation:c:bis} and \eqref{eq:c_n}.
\end{proof}

\section*{Acknowledgements} 

The authors were supported by a Trond Mohn Foundation grant.

	\bibliography{biblio}
	\bibliographystyle{plain}
	
	
\end{document}